\documentclass{amsart}
\usepackage{mathtools, amsmath, amssymb, amsthm, url, enumerate}
\usepackage{tikz}
\usepackage{comment}
\usetikzlibrary{cd}

\usepackage[margin=25mm]{geometry}
 
\usepackage[all, cmtip]{xy}
\usepackage[hypertexnames=false]{hyperref}
 
\newcommand{\R}{\mathbb{R}}
\newcommand{\Z}{\mathbb{Z}}

\newcommand{\Q}{\mathbb{Q}}

\newcommand{\vol}{\mathrm{vol}}
\newcommand{\Clin}{C^{\infty}_{\mathrm{lin}}}
\newcommand{\X}{\mathcal{X}}
\newcommand{\Tor}{\mathbb{T}}

\DeclareMathOperator{\D}{\mathcal{D}}

\DeclareMathOperator{\Path}{Path}

\DeclareMathOperator{\pr}{pr}
\DeclareMathOperator{\ad}{ad}

\DeclareMathOperator{\id}{id}
\DeclareMathOperator{\evl}{ev}

\DeclareMathOperator{\Ker}{Ker}
\DeclareMathOperator{\colim}{colim}

\theoremstyle{definition}
 
\newtheorem{thm}{Theorem}[section]
\newtheorem{defn}[thm]{Definition}
\newtheorem{rem}[thm]{Remark}
\newtheorem{prop}[thm]{Proposition}
\newtheorem{lem}[thm]{Lemma}

\newtheorem{ex}[thm]{Example}

\newtheorem*{thm*}{Theorem}
\newtheorem*{defn*}{Definition}

\newtheorem{notation}[thm]{Notation}
\newtheorem{convention}[thm]{Convention}
\newtheorem{mainthm}{Theorem}

\title{On Riemannian geometry in elastic diffeology}
\author[Y. Shiobara]{Yusuke Shiobara}
\address{%
  Department of Science and Technology,
  Graduate School of Medicine, Science and Technology,
  Shinshu University,
  Matsumoto, Nagano 390-8621, Japan
}
\email{24hs603d@shinshu-u.ac.jp}

\subjclass[2020]{Primary 58A40; Secondary 18F15, 18F40, 53C05, 53C22, 58D15, 58B20}
\keywords{diffeology, tangent category, connection, geodesic, energy functional, irrational torus, mapping space.}

\begin{document}
\begin{abstract}
  Blohmann's tangent functor makes the category of elastic diffeological
  spaces a Cartesian tangent category in the sense of Cockett and
  Cruttwell. We use this structure to develop Riemannian geometry on
  elastic diffeological spaces. Following Cockett and Cruttwell, with
  additional axioms for scalar multiplication, we introduce connections,
  covariant derivatives and curvature; for a Riemannian metric we prove a
  Koszul formula and the uniqueness of the Levi-Civita covariant
  derivative. Our main result is a variational characterization of
  geodesics: if a full connection whose vertical part is Levi-Civita admits
  a geodesic flow fixing the zero section, then the critical paths of the
  energy functional are exactly the paths satisfying the geodesic equation,
  and regular length-minimizing paths are geodesics after constant-speed
  reparametrization. On elastic spaces for which the real line is a curve
  object, completeness of the geodesic spray yields such a flow and
  uniqueness of geodesics. The results apply to quotients of complete
  Riemannian manifolds by free isometric actions of countable groups, such
  as irrational tori, and, under a tangent-commuting hypothesis, to mapping
  spaces from a closed manifold into elastic spaces for which the real line
  is a curve object, including free loop spaces of complete Riemannian
  manifolds.
\end{abstract}
\maketitle

\section{Introduction}\label{sec:intro}

Diffeology, introduced by Souriau and developed systematically by Iglesias-Zemmour~\cite{PIZ13}, provides a framework for smooth geometry that contains the category of manifolds as a full subcategory and accommodates singular quotients and infinite-dimensional mapping spaces. Its topological side, including de~Rham theory and homotopy theory~\cite{Kihara}, is well developed; Riemannian geometry is much less so, partly because several inequivalent notions of tangent space are in use~\cite{Taho}. Iglesias-Zemmour~\cite{PIZ25} introduced a notion of Riemannian metric that avoids tangent spaces altogether, and Kuribayashi, Sakai, and the author~\cite{KSS} interpreted it through Blohmann's tangent functor~\cite{B23}, a colimit of the tangent bundles of domains.

Connections, curvature, geodesics, and variational principles require more than a tangent functor. A \emph{tangent category}~\cite{R, CC14} is a category with an endofunctor $T$, thought of as the tangent bundle functor, and natural transformations playing the roles of the bundle projection $\pi\colon T \to 1$, the zero section, fibrewise addition, the vertical lift $\lambda\colon T \to T^2$, and the canonical flip $\tau\colon T^2 \to T^2$, subject to the identities that these maps satisfy for smooth manifolds; the definition is recalled in Section~\ref{subsec:tangent-categories}. Connections and curvature in tangent categories were developed by Cockett and Cruttwell~\cite{CC17}, and dynamical systems, their solutions, and curve objects, which play the role of the real line, were developed by Cockett, Cruttwell, and Lemay~\cite{CCL}, providing an abstract formulation of differential equations. Blohmann~\cite{B23} showed that the category $\mathsf{Elast}$ of elastic diffeological spaces, which contains all finite-dimensional manifolds, all convenient vector spaces, and the manifolds modelled on them, is a Cartesian tangent category with scalar $\R$-multiplication (Theorem~\ref{thm:tangent-structure-2}). We use this structure to develop Riemannian geometry on elastic diffeological spaces, focusing on the variational theory of geodesics.

\subsection*{The main results}
Two features of the setting shape the results. First, $\mathsf{Elast}$ contains the convenient vector spaces, on which smooth differential equations need not have unique solutions; accordingly, $\R$ is not a curve object in $\mathsf{Elast}$. Second, elastic diffeological spaces in general have neither charts nor a local existence theorem for differential equations, and an exponential map cannot be assumed a priori; so the classical proof that critical points of the energy are geodesics, which constructs variations in coordinates or through the exponential map, does not apply directly.

Let $(M, g)$ be an elastic Riemannian diffeological space, that is, an elastic diffeological space with a Riemannian metric (Definition~\ref{def:metric}). For $x, y \in M$ let $\Path(M; x, y)$ be the diffeological space of smooth paths $\gamma\colon \R \to M$ with $\gamma(0) = x$ and $\gamma(1) = y$, and let
\[
  E(\gamma) = \frac{1}{2}\int_0^1 g\bigl(\dot\gamma(t), \dot\gamma(t)\bigr)\, dt
\]
be its energy. A path is a \emph{critical path} if the derivative of $E$ vanishes along every smooth variation in $\Path(M; x, y)$ (Definition~\ref{def:critical-path}). Geodesics are defined through the covariant derivative of a connection (Definition~\ref{def:geodesic}). A full connection $(K, H)$ has a geodesic spray $S$, and a \emph{geodesic flow} is a smooth map $\Phi \colon \R \times TM \to M$ solving the second-order equation defined by $S$ for all initial velocities simultaneously (Definition~\ref{def:geodesic-spray}). We say that $\Phi$ \emph{fixes the zero section} if $\Phi(s, 0_M(x)) = x$ for all $s \in \R$ and $x \in M$ (Definition~\ref{def:fix-zero}).

\begin{mainthm}\label{mainthm:A}
Let $(M, g)$ be an elastic Riemannian diffeological space, and let $(K, H)$ be a full connection on $M$ whose vertical part $K$ is a Levi-Civita connection for $g$ and which admits a geodesic flow fixing the zero section.
\begin{enumerate}[(i)]
  \item A path $\gamma \in \Path(M; x, y)$ is a critical path of $E$ if and only if it satisfies the geodesic equation $\nabla^K_\partial \dot\gamma = 0$ on $[0,1]$.
  \item If $\gamma \in \Path(M; x, y)$ is regular and minimizes length among all paths from $x$ to $y$, then its constant-speed reparametrization satisfies the geodesic equation on $[0,1]$.
\end{enumerate}
\end{mainthm}

The hypothesis of Theorem~\ref{mainthm:A} concerns a particular geodesic flow and depends on the metric and the connection. Our second theorem obtains it, together with the uniqueness of geodesics, from a condition on the space alone, by applying the theory of dynamical systems and curve objects of Cockett, Cruttwell, and Lemay~\cite{CCL}. Let $\mathsf{Elast}_{\mathrm{curve}}$ be the full subcategory of those elastic spaces $M$ such that, for every $n \ge 0$, each parametrized dynamical system on $T^nM$ has at most one solution with time domain $\R$ (Definition~\ref{def:elast-lc}); it contains all finite-dimensional and Banach manifolds.

\begin{mainthm}\label{mainthm:B}
Let $M \in \mathsf{Elast}_{\mathrm{curve}}$, and let $(K, H)$ be a full connection on $M$ whose geodesic spray is complete. Then $(K, H)$ admits a geodesic flow $\Phi$, every geodesic flow of $(K, H)$ fixes the zero section, and for every $v \in TM$ there is exactly one geodesic $\gamma$ with $\dot\gamma(0) = v$, namely $\gamma(t) = \Phi(t, v)$.
\end{mainthm}

The proofs are assembled in Section~\ref{subsec:proof-main}. The essential step is the converse direction of Theorem~\ref{mainthm:A}\,(i): every vector field along a path that vanishes at the endpoints must be realized as the variation field of a genuine smooth variation in the functional diffeology of $\Path(M; x, y)$. Without using charts, we construct such variations from the geodesic flow, as $(s, t) \mapsto \Phi(s, W(t))$ (Lemma~\ref{lem:realization}); that $\Phi$ fixes the zero section guarantees that they fix the endpoints. Together with the first variation formula (Lemma~\ref{lem:first-variation-formula}), which rests on the functional diffeology of path spaces and on integration along paths---structures not provided by the axioms of a general tangent category---this yields the variational characterization.

The condition defining $\mathsf{Elast}_{\mathrm{curve}}$ is independent of the metric and the connection, and it is stable under the tangent functor and under quotients by principal actions of diffeologically discrete groups~\cite{ShioQ}. Theorem~\ref{mainthm:B} uses it to obtain the uniqueness of geodesics and the zero-section property, thereby providing a sufficient condition for Theorem~\ref{mainthm:A}; in our examples this is how the hypothesis of Theorem~\ref{mainthm:A} is verified, directly for quotients and through the target for mapping spaces (Theorem~\ref{thm:mapping-variational}). It is not a characterization of the spaces that carry a geodesic flow: in Theorem~\ref{mainthm:B}, completeness of the geodesic spray is an additional hypothesis, ensuring a global geodesic flow for all initial velocities.

In this approach the connection plays two roles: it defines the geodesic equation and the geodesic flow, and the flow in turn supplies the variations needed in the proof of Theorem~\ref{mainthm:A}. Classically, the geodesic flow can also be described as a Hamiltonian flow on the cotangent bundle, transported to the tangent bundle by the metric. Diffeological analogues of the cotangent bundle have been studied, notably the pseudo-bundle of values of differential $1$-forms~\cite{PIZ13, Per17}, but to our knowledge a Hamiltonian description of the geodesic flow compatible with Blohmann's tangent functor is not available; moreover, the metrics considered here need not induce an isomorphism between tangent and cotangent spaces. The connection-based approach requires no such identification, and it provides the covariant derivatives used in the first variation formula. We therefore work with connections throughout.

\subsection*{Examples}
For finite-dimensional Riemannian manifolds with complete geodesic spray, Theorems~\ref{mainthm:A} and~\ref{mainthm:B} recover classical facts; their interest lies in spaces that are not manifolds. Using the descent results of a companion paper~\cite{ShioQ}, our first examples are quotients of geodesically complete finite-dimensional Riemannian manifolds by free isometric actions of countable groups (Proposition~\ref{prop:quotient-main}). In particular, the irrational torus $\Tor_\alpha = \R/(\Z + \alpha\Z)$, $\alpha \notin \Q$, has a complete geodesic flow although its D-topology is indiscrete. The pseudo-distance induced by its metrics vanishes identically, yet any two distinct points are joined by infinitely many geodesics, all critical for the energy, whose lengths accumulate at $0$ (Example~\ref{ex:torus-critical}); the variational theory retains information invisible to the pseudo-distance.

Mapping spaces form the second class of examples. For a closed manifold $M$ and an elastic Riemannian diffeological space $N$, under a tangent-commuting hypothesis that holds when $N$ is a manifold, all of the above structures on $C^\infty(M,N)$ are obtained pointwise from those on $N$, and the geodesic spray of $C^\infty(M,N)$ is complete whenever that of $N$ is (Theorem~\ref{thm:summary-of-mapping-space}). If moreover $N \in \mathsf{Elast}_{\mathrm{curve}}$, the geodesic flow of $C^\infty(M,N)$ fixes the zero section, so Theorem~\ref{mainthm:A} applies to $C^\infty(M,N)$, although we do not show that $C^\infty(M,N)$ itself belongs to $\mathsf{Elast}_{\mathrm{curve}}$ (Theorem~\ref{thm:mapping-variational}). This includes free loop spaces of geodesically complete finite-dimensional Riemannian manifolds with the $L^2$ metric, and for $N = \R^n$ part~(ii) of Theorem~\ref{mainthm:A} is not vacuous (Example~\ref{ex:mapping-flat}).

\subsection*{Relation to previous work}
Connections and curvature in tangent categories are due to Cockett and Cruttwell~\cite{CC17} and Lucyshyn-Wright~\cite{LW}, and dynamical systems, curve objects, and the geodesic spray to Cockett, Cruttwell, and Lemay~\cite{CCL}; we add two axioms for scalar multiplication (Remark~\ref{rem:scalar-axiom}). Sections~\ref{sec:connections}--\ref{sec:curvature} adapt this theory and the classical arguments for metric compatibility, the Koszul formula, and the uniqueness of the Levi-Civita covariant derivative to the elastic setting; results from the literature are stated without proof. The existence of a Levi-Civita connection is not known in general (Remark~\ref{rem:lc-existence}). The new results of this paper are the variational theory of Section~\ref{sec:geodesics}, summarized in Theorem~\ref{mainthm:A}, and its applications in Section~\ref{sec:examples}, notably to mapping spaces (Theorem~\ref{thm:mapping-variational}). Theorem~\ref{mainthm:B} applies the theory of dynamical systems and curve objects of~\cite{CCL} to the class $\mathsf{Elast}_{\mathrm{curve}}$; the geodesic spray and the vanishing of the acceleration of the geodesic flow are from~\cite[Section~3.5]{CCL}, while the uniqueness of geodesics and the zero-section property are derived here.

Riemannian geometry on diffeological spaces has also been studied by Goldammer and Welker~\cite{GW}, motivated by optimization. They work with Vincent's tangent spaces~\cite{V}, which differ from Blohmann's tangent functor, and define a Levi-Civita connection directly as a covariant derivative, whereas we start from a connection map in the sense of tangent categories; the induced covariant derivatives satisfy the same axioms (Propositions~\ref{prop:torsion-free-bracket}, \ref{prop:covariant-derivative-properties}, and~\ref{prop:metric-compatibility-vf}). In their setting, too, only the uniqueness of the Levi-Civita connection is known. 

\subsection*{Organization}
Section~\ref{sec:tangent-structure} recalls diffeological spaces, tangent categories, Blohmann's elastic diffeological spaces, and the calculus of vector fields. Sections~\ref{sec:connections}, \ref{sec:riemannian-metric}, and~\ref{sec:curvature} treat connections, Riemannian metrics, and curvature. Section~\ref{sec:geodesics} develops geodesics and the energy functional and proves Theorems~\ref{mainthm:A} and~\ref{mainthm:B}. Section~\ref{sec:examples} discusses irrational tori and mapping spaces.

\subsection*{Setting and conventions}
Throughout, $\mathsf{Elast}$ is Blohmann's category of elastic diffeological spaces~\cite{B23} with his tangent functor $T$; manifolds carry their standard diffeology, on which $T$ is the classical tangent functor. Composition is written in the applicative order $g \circ f$, whereas \cite{CC17, CCL} use the diagrammatic order. Riemannian metrics are definite (Definition~\ref{def:metric}). Standing assumptions are stated at the head of the section or subsection concerned.

\subsection*{Acknowledgement}
I would like to express my sincere gratitude to my supervisor, Keiichi Sakai, for his valuable guidance and continuous support since my undergraduate years. I am also grateful to my sub-advisor, Katsuhiko Kuribayashi, for many insightful discussions and helpful advice. I would like to thank David Miyamoto for his generous support and encouragement.

This work was supported by JST SPRING, Grant Number JPMJSP2144, at Shinshu University.

\subsection*{Declaration of generative AI in the writing process}
During the preparation of this manuscript, the author used ChatGPT and Claude for language editing and for discussing possible formulations, presentations, and future research directions. All mathematical results, proofs, and scientific conclusions were developed and verified by the author. The author carefully reviewed and revised all AI-generated suggestions and takes full responsibility for the content of this publication.

\section{Preliminaries}\label{sec:tangent-structure}

\subsection{Diffeological spaces}\label{subsec:diffeology}

We recall the basic notions of diffeology; see \cite{PIZ13} for details and
\cite[Section~2]{KSS} for a more extensive summary.

\begin{defn}[{\cite[1.5, 1.14]{PIZ13}}]\label{def:diffeology}
  A \textbf{domain} is an open subset of a Euclidean space of finite
  dimension, and a \textbf{parametrization} of a set $X$ is a map
  $P\colon U_P \to X$ from a domain $U_P$. A \textbf{diffeology} on $X$ is a
  set $\D(X)$ of parametrizations, called \textbf{plots}, such that:
  \begin{description}
    \item[Covering] every constant parametrization is a plot;
    \item[Locality] a parametrization $P$ is a plot if every point of $U_P$
      has an open neighbourhood $V$ such that $P|_V$ is a plot;
    \item[Smooth compatibility] if $P$ is a plot and $f\colon V \to U_P$ is a
      smooth map between domains, then $P \circ f$ is a plot.
  \end{description}
  A \textbf{diffeological space} is a set with a diffeology. A map
  $f\colon X \to Y$ between diffeological spaces is \textbf{smooth} if
  $f \circ P$ is a plot of $Y$ for every plot $P$ of $X$. Diffeological spaces
  and smooth maps form a category $\mathsf{Diff}$.
\end{defn}

\begin{ex}\label{ex:diffeologies}
  \begin{enumerate}[(i)]
    \item \emph{Manifolds.} The smooth maps from domains to a manifold $M$
      form its \textbf{standard diffeology}; this identifies the category of
      manifolds with a full subcategory of $\mathsf{Diff}$. The
      \textbf{discrete diffeology} on a set consists of the locally constant
      parametrizations; a group with this diffeology is called
      \textbf{diffeologically discrete}.
    \item \emph{Subspaces.} For a subset $A$ of a diffeological space $X$,
      the parametrizations of $A$ that are plots of $X$ form the
      \textbf{subset diffeology}. An injective smooth map $j\colon A \to X$ is an
      \textbf{induction} if a parametrization $P$ of $A$ is a plot whenever
      $j \circ P$ is a plot of $X$.
    \item \emph{Quotients.} For an equivalence relation on $X$ with quotient
      map $q\colon X \to X/{\sim}$, a parametrization $P$ of $X/{\sim}$ is a plot
      of the \textbf{quotient diffeology} if every point of $U_P$ has an open
      neighbourhood $V$ with $P|_V = q \circ Q$ for some plot $Q$ of $X$. A
      surjective smooth map $\pi\colon X \to Y$ is a \textbf{subduction} if the
      diffeology of $Y$ is obtained from $\pi$ in this way.
    \item \emph{Products and mapping spaces.} A parametrization of
      $\prod_i X_i$ is a plot of the \textbf{product diffeology} if each of its
      components is a plot. A parametrization $P$ of the set $C^\infty(X, Y)$
      of smooth maps is a plot of the \textbf{functional diffeology} if its
      \textbf{adjoint}
      \begin{align*}
        \ad(P)\colon U_P \times X \to Y, \qquad (u, x) \mapsto P(u)(x),
      \end{align*}
      is smooth; we use this correspondence in both directions. When $X$ is a
      domain, $\ad(P)$ is a plot of $Y$. With these structures,
      $\mathsf{Diff}$ is complete, cocomplete, and Cartesian
      closed~\cite[Section~2]{KSS}.
  \end{enumerate}
\end{ex}

The \textbf{D-topology} of a diffeological space $X$ is the finest topology
for which all plots are continuous~\cite[2.8]{PIZ13}; smooth maps are
continuous for the D-topologies.

A set $\mathcal{G}$ of parametrizations of $X$ \textbf{generates} the
diffeology consisting of those parametrizations $P$ such that every point of
$U_P$ has an open neighbourhood $V$ on which $P|_V$ is constant or of the form
$Q \circ f$ with $Q \in \mathcal{G}$ and $f\colon V \to U_Q$ smooth; a subset
of $\D(X)$ generating $\D(X)$ is a \textbf{generating family}~\cite[1.66]{PIZ13}.
For instance, the charts of an atlas generate the standard diffeology of a
manifold.

A \textbf{differential $k$-form} $\omega$ on $X$ is a family
$(\omega_P)_{P \in \D(X)}$ of $k$-forms $\omega_P \in \Omega^k(U_P)$ such that
$f^*\omega_P = \omega_{P \circ f}$ for every plot $P$ and every smooth map
$f\colon V \to U_P$ between domains; that is,
$\Omega^k(X) = \lim_{P \in \D(X)} \Omega^k(U_P)$, where plots are regarded as
the objects of a category as in Section~\ref{subsec:elastic}
(cf.\ \cite[Chapter~6]{PIZ13}). The exterior derivative is defined plotwise,
$(d\omega)_P := d(\omega_P)$; for $f \in C^\infty(X, \R)$, the $1$-form $df$
is given by $(df)_P = d(f \circ P)$. The two constructions (iii) and (iv) provide
the principal examples of Section~\ref{sec:examples}: the irrational torus
$\R/(\Z + \alpha\Z)$, $\alpha \notin \Q$, with the quotient diffeology, whose
D-topology is indiscrete, and mapping spaces $C^\infty(M, N)$ with the
functional diffeology. The functional diffeology is also what makes smooth
variations of paths, used in Section~\ref{sec:geodesics}, into plots of path
spaces.

\subsection{Tangent categories}\label{subsec:tangent-categories}

A tangent category axiomatizes the tangent bundle functor of smooth manifolds.
We recall the notion of a tangent category with negatives in the sense of
Rosick\'y~\cite{R}, following the formulation of Aintablian and
Blohmann~\cite[Section~2]{AB} and the exposition in \cite[Appendix~A]{M}, and
Cartesian tangent categories with scalar multiplication~\cite{AB}. Only the
ingredients needed for connections and Riemannian geometry are reviewed; for
the general theory we refer to \cite{AB, CC14, CC15, M}. In this subsection the
objects of an abstract category $\mathcal{C}$ are denoted by $X, Y, Z$.

\subsubsection{Bundles and symmetric structures}\label{sec:first-definitions}

Following \cite[Section~2.1]{AB} and \cite[Appendix~A.1]{M}, consider a
category $\mathcal{C}$ that has a terminal object $*$ and such that for any
object $X$ the overcategory $\mathcal{C}\downarrow X$ has finite products.

\begin{defn}\label{def:bundle-abelian-groups}
  A \textbf{bundle of abelian groups} $p\colon A \to X$ is an abelian group
  object in $\mathcal{C}\downarrow X$, that is, $p$ together with an addition
  $+\colon A \times_X A \to A$, a zero section $0\colon X \to A$, and an inverse
  $A \to A$ satisfying the axioms of an abelian group; no local triviality is
  assumed. A \textbf{morphism of bundles} from $p\colon A \to X$ to
  $p'\colon A' \to X'$ is a pair $(\phi, f)$ of morphisms with
  $p' \circ \phi = f \circ p$; it is a \textbf{morphism of bundles of abelian
  groups} if moreover $\phi \circ + = +' \circ (\phi \times_f \phi)$. For a
  commutative ring object $R$, a \textbf{bundle of $R$-modules} is a bundle of
  abelian groups with a scalar multiplication $R \times A \to A$ satisfying the
  axioms of a module.
\end{defn}

\begin{defn}\label{def:symmetric-structure}
  Let $T\colon \mathcal{C} \to \mathcal{C}$ be an endofunctor. A natural
  transformation $\tau\colon T^2 \to T^2$ is a \textbf{symmetric structure} if
  $\tau \circ \tau = 1$ and
  $\tau_{12} \circ \tau_{23} \circ \tau_{12} = \tau_{23} \circ \tau_{12} \circ
  \tau_{23}$, where $\tau_{12} \coloneqq \tau T$ and
  $\tau_{23} \coloneqq T \tau$. The functor $T$ \textbf{preserves the fibre
  products} of a bundle $p\colon A \to X$ if the natural morphisms
  $\nu_{k,X} \colon T(A \times_X \dots \times_X A) \to TA \times_{TX} \dots
  \times_{TX} TA$ are isomorphisms for all $k \geq 1$.
\end{defn}

\subsubsection{Tangent categories with negatives}\label{sec:RosickysAxioms}

\begin{defn}\label{def:tangent-structure-axioms}
  A tangent structure on a category $\mathcal{C}$ consists of a covariant endofunctor $T\colon \mathcal{C} \to \mathcal{C}$ and the following natural transformations:
  \makeatletter
  $$
  \begin{alignedat}{2}
    & \text{projection}\quad &  \pi &\colon T \to 1 \\
    &\text{zero-section}\quad & 0 &\colon 1 \to T \\
    & \text{addition}\quad & +&\colon T_2 \to T \\
    &\text{vertical lift}\quad & \lambda&\colon T \to T^2 \\
    &\text{symmetric structure}\quad & \tau &\colon T^2 \to T^2.
  \end{alignedat}
  $$
  \makeatother
  The axioms they must satisfy are as follows:
  \begin{itemize}
    \item The projection $\pi\colon T \to 1$ is a bundle of abelian groups over $1$ with zero-section $0$ and addition $+$.
    \item All pullbacks $T_k \coloneqq T \times_1 \dots \times_1 T$ exist and are preserved pointwise by $T$.
    \item The natural transformation $\tau$ is a symmetric structure and a morphism of bundles of abelian groups.
    \item The following diagrams commute, and in particular, the left diagram is a morphism of bundles of abelian groups:
      $$
        \begin{tikzcd}
          T  \ar[r, "\lambda"] \ar[d, "\pi"'] & T^2 \ar[d, "\pi T"] \\
          1  \ar[r, "0"'] & T
        \end{tikzcd}
        \qquad
        \begin{tikzcd}
          T \ar[r, "\lambda"] \ar[d, "\lambda"'] & T^2   \ar[d, "\lambda T"] \\
          T^2  \ar[r, "T\lambda"'] & T^3
        \end{tikzcd}
      $$
    \item The following diagrams commute with respect to the vertical lift and the symmetric structure:
      $$
        \begin{tikzcd}
          & T \ar[dl, "\lambda"'] \ar[dr, "\lambda"] & \\
          T^2  \ar[rr, "\tau"'] & & T^2
        \end{tikzcd}
        \qquad
        \begin{tikzcd}
          T^2  \ar[r, "T\lambda"] \ar[d, "\tau"'] & T^3  \ar[r, "\tau T"] & T^3 \ar[d, "T\tau"] \\
          T^2  \ar[r, "\lambda T"'] & T^3 \ar[r, "\tau"'] & T^3
        \end{tikzcd}
      $$
    \item The diagram formed by the following natural transformation $\lambda_2$ is a pointwise pullback:
      $$
        \begin{tikzcd}
          T_2  \ar[r, "T0 \times_0 \lambda"] \ar[d, "\pi \circ \pr_1"']
            & T_2T \ar[r, "+_T"]
            & T^2 \ar[r, "\tau"]
            & T^2 \ar[d, "T\pi"] \\
          1  \ar[rrr, "0"'] & & & T
        \end{tikzcd}
      $$
  \end{itemize}
\end{defn}

The first axiom requires the fibres of $\pi$ to be abelian groups. In the
terminology of Cockett and Cruttwell~\cite{CC14}, where these fibres are a
priori only commutative monoids,
Definition~\ref{def:tangent-structure-axioms} describes a tangent category with
negatives (cf.\ \cite[Section~2]{AB}).

\subsubsection{Cartesian tangent categories and scalar multiplication}\label{sec:CatTanScal}

Suppose the tangent category $\mathcal{C}$ has finite products.

\begin{defn}\label{def:cartesian-tangent-category}
  A tangent category $(\mathcal{C}, T)$ is called a Cartesian tangent category if $T* \cong *$ and the natural morphisms
  $$
    \chi_{X,Y}\colon T(X \times Y) \longrightarrow TX \times TY
  $$
  are isomorphisms for any objects $X, Y$.
\end{defn}

In a Cartesian tangent category, for a morphism $f\colon X \times Y \to Z$, the partial tangent morphisms with respect to each variable can be defined as follows:
$$
\begin{aligned}
  T_{(1)} f \colon TX \times Y
    &\xrightarrow{\id_{TX} \times 0_Y} TX \times TY
     \xrightarrow{\chi_{X,Y}^{-1}} T(X \times Y)
     \xrightarrow{Tf} TZ, \\
  T_{(2)} f \colon X \times TY
    &\xrightarrow{0_X \times \id_{TY}} TX \times TY
     \xrightarrow{\chi_{X,Y}^{-1}} T(X \times Y)
     \xrightarrow{Tf} TZ.
\end{aligned}
$$

\begin{prop}[{\cite[Proposition 2.10]{CC14}}]\label{prop:decomposition-tangent-morphism}
  For a morphism $f\colon X \times Y \to Z$ in a Cartesian tangent category $(\mathcal{C}, T)$, the following holds:
  $$
  \begin{aligned}
    Tf = T_{(1)} f \circ (\id_{TX} \times \pi_Y) \circ \chi_{X,Y}
         +_{TZ} T_{(2)} f \circ (\pi_X \times \id_Y) \circ \chi_{X,Y}.
  \end{aligned}
  $$
\end{prop}

Fix a commutative ring object $R$ in a Cartesian tangent category $\mathcal{C}$.
The natural transformation $1 \times R \to 1$ inherits the structure of a ring object in the category $\operatorname{End}(\mathcal{C}) \downarrow 1$ from $R$.
An $R$-module structure on $T$ is a $(1 \times R \to 1)$-module structure on $\pi$.
This is given by a natural morphism $\kappa\colon R \times T \to T$.
Furthermore, if the natural morphism $TR \to R \times R$ is an isomorphism of $R$-modules, $R$ is called tangent-stable.

\begin{defn}[{\cite[Definition~2.24]{AB}}]\label{def:scalar-multiplication}
  A Cartesian tangent category with scalar $R$-multiplication is a tuple consisting of a Cartesian tangent category $\mathcal{C}$, a commutative ring object $R$, and a tangent-stable $R$-module structure $\kappa$ on $T$ that makes the following diagrams commute:
  $$
    \begin{tikzcd}[column sep=large]
      R \times TX \ar[r, "\id_R \times \lambda_X"] \ar[d, "\kappa_X"']
        & R \times T^2 X \ar[d, "\kappa_{TX}"] \\
      TX \ar[r, "\lambda_X"'] & T^2 X
    \end{tikzcd}
  $$
  $$
    \begin{tikzcd}[column sep=6em]
      TR \times TX \ar[r, "T_{(1)}\kappa_X"] \ar[d, "{(\pi_R, \mathsf{Val}_R) \times \id_{TX}}"', "\cong"]
        &[4em] T^2X \\
      R \times R \times TX
        \ar[r, "{(\kappa_X \circ (\pr_1,\pr_3),\, \kappa_X \circ (\pr_2,\pr_3))}"']
        & T_2 X \ar[u, "\lambda_{2,X}"']
    \end{tikzcd}
  $$
  $$
    \begin{tikzcd}[column sep=3em]
      R \times T^2 X \ar[r, "T_{(2)}\kappa_X"] \ar[d, "\id_R \times \tau_X"'] & T^2 X \\
      R \times T^2 X \ar[r, "\kappa_{TX}"'] & T^2 X \ar[u, "\tau_X"']
    \end{tikzcd}
  $$
  Here, $\mathsf{Val}_R \colon TR \to R$ is the projection onto the second component of $TR \cong R \times R$.
\end{defn}

\subsubsection{Differential bundles}\label{sec:diff-bundles}

A \textbf{differential bundle}~\cite{CC18} over an object $M$ is a bundle of
commutative monoids $q\colon E \to M$, with addition $+_q$ and zero section $z_q$, equipped
with a lift $\lambda_q\colon E \to TE$ subject to the axioms of~\cite{CC18}; we
write $\mathsf{q} = (q, +_q, z_q, \lambda_q)$. The tangent bundle
$(\pi_M, +_M, 0_M, \lambda_M)$ is the basic example.

\begin{defn}[\cite{CC18}]\label{def:differential-bundle-morphism}
  Let $\mathsf{q} = (q, +_q, z_q, \lambda_q)$ over $M$ and
  $\mathsf{q}' = (q', +_{q'}, z_{q'}, \lambda_{q'})$ over $M'$ be differential
  bundles. A \textbf{bundle morphism} $(f_1, f_0)\colon \mathsf{q} \to
  \mathsf{q}'$ is a pair of morphisms $f_1 \colon E \to E'$ and
  $f_0 \colon M \to M'$ with $q' \circ f_1 = f_0 \circ q$. It is
  \textbf{linear} if moreover $\lambda_{q'} \circ f_1 = T(f_1) \circ \lambda_q$.
\end{defn}

By \cite[Proposition~2.16]{CC18}, a linear bundle morphism is automatically
additive: it preserves the addition and the zero section.

\subsection{Elastic diffeological spaces}\label{subsec:elastic}

We use the notions of Section~\ref{subsec:diffeology}. The purpose of this subsection is to fix notation for
Blohmann's diffeological tangent functor and for the class of \emph{elastic}
spaces on which it carries a tangent structure in the sense of
Rosick\'y--Cockett--Cruttwell. All results in this subsection are due to Blohmann
\cite{B23,B24}; we state them without proof, in the form and notation in which
they will be used in Sections~\ref{sec:connections}--\ref{sec:examples}.
A detailed exposition, to which we are indebted, can also be found in
\cite[Section~2.1]{M}.

We denote by $\mathsf{Euc}$ the category whose objects are open subsets of
Euclidean spaces and whose morphisms are smooth maps, and let
$\mathcal{Y}\colon \mathsf{Euc} \to \mathsf{Diff}$ be the full and faithful
functor assigning to each domain its standard diffeology. On $\mathsf{Euc}$ one
has the classical tangent functor
$\hat{T}U = U \times \R^{\dim U}$,
$\hat{T}f(u,v) = (f(u), D_uf(v))$.

\begin{defn}[{\cite[Section~4.1]{B23}}]\label{def:tangent-functor}
For an endofunctor $\hat{F}\colon \mathsf{Euc} \to \mathsf{Euc}$ and a natural
transformation $\hat{\alpha}\colon \hat{F} \to \hat{G}$, set
\begin{equation*}
  \mathbb{L}\hat{F} \coloneqq \mathrm{Lan}_{\mathcal{Y}}\,\mathcal{Y}\hat{F},
  \qquad
  \mathbb{L}\hat{\alpha} \coloneqq
  \mathrm{Lan}_{\mathcal{Y}}\,\mathcal{Y}\hat{\alpha}.
\end{equation*}
The \textbf{tangent functor} $T \colon \mathsf{Diff} \to \mathsf{Diff}$ is the
left Kan extension $T \coloneqq \mathbb{L}\hat{T}$.
\end{defn}

Because $\mathsf{Euc}$ is small and $\mathsf{Diff}$ is cocomplete, these Kan
extensions are computed pointwise: regarding the diffeology $\D(M)$ of a space
$M$ as a category (objects: plots $P\colon U_P \to M$; morphisms: smooth maps
$h\colon U_P \to U_Q$ with $Q \circ h = P$), one has
\begin{equation*}
  (\mathbb{L}\hat{F})M \;=\; \colim_{P \in \D(M)} \hat{F}U_P ,
  \qquad\text{in particular}\qquad
  TM \;=\; \colim_{P \in \D(M)} \bigl(U_P \times \R^{\dim U_P}\bigr).
\end{equation*}

\begin{rem}\label{rem:tangent-structure-2}
We denote by $[P, w]$ the class in $(\mathbb{L}\hat{F})M$ of an element
$w \in \hat{F}U_P$; thus a tangent vector is written $[P,u,v]$ with
$(u,v) \in U_P \times \R^{\dim U_P}$. The identification is the equivalence
relation \emph{generated} by
\begin{equation*}
  [P \circ h,\; w] \;=\; [P,\; \hat{F}h(w)]
  \qquad\text{for every morphism of plots } h ,
\end{equation*}
so that $[P,w] = [P',w']$ precisely when the two representatives are joined by
a finite chain of such elementary moves (in either direction) over
intermediate plots of $M$. On classes, the Kan extension of a natural
transformation acts representative-wise,
$(\mathbb{L}\hat{\alpha})_M[P,w] = [P, \hat{\alpha}_{U_P}(w)]$, and the tangent
map of a smooth map $f \colon M \to N$ is
\begin{equation*}
  T(f)(\,[P,u,v]) = [\,f \circ P,\, u,\, v\,].
\end{equation*}
\end{rem}

The Euclidean tangent bundle carries its usual structure maps
$\hat{\pi}(u,v) = u$, $\hat{0}(u) = (u,0)$, and
$\hat{\kappa}(t,(u,v)) = (u,tv)$, all natural in $U$; their Kan extensions
therefore exist on all of $\mathsf{Diff}$. For $\hat{\kappa}$ a word of caution
is in order, since its source $\R \times \hat{T}$ involves a product: the
identification $\bigl(\mathbb{L}(\R \times \hat{T})\bigr)M \cong \R \times TM$
used below is legitimate because $\mathsf{Diff}$ is Cartesian closed, so that
$\R \times (-)$ is a left adjoint and preserves the colimit defining $TM$. By
contrast, the \emph{fibre} products entering the addition and the flip do not
commute with the colimit; this is the source of the difficulties addressed by
the elasticity axioms below.

\begin{defn}\label{def:nat-trans}
The \textbf{projection} $\pi \coloneqq \mathbb{L}\hat{\pi}$,
\textbf{zero section} $0 \coloneqq \mathbb{L}\hat{0}$, and
\textbf{scalar multiplication} $\kappa \coloneqq \mathbb{L}\hat{\kappa}$ act
on classes by
\begin{equation*}
  \pi_M[P,u,v] = P(u), \qquad
  0_M(x) = [P,u,0] \quad (\text{any } P \text{ with } P(u)=x), \qquad
  \kappa_M(t,[P,u,v]) = [P,u,tv].
\end{equation*}
The value $0_M(x)$ is independent of the chosen plot: if $P(u) = P'(u') = x$,
then both classes agree with $[\,\mathrm{const}_x,\,0,\,0\,]$ by naturality of
$\hat{0}$ applied to the constant plots factoring through $P$ and $P'$.
\end{defn}

The tangent bundle is generated by curves in the following sense; this is the
reason why, in the variational theory of Section~\ref{sec:geodesics}, tangent
vectors to path spaces can be realized by genuine variations.

\begin{defn}\label{def:delta}
  Define $\delta = \delta_M \colon C^{\infty}(\mathcal{Y}\R, M) \to TM$ by
  $\delta(\gamma) \coloneqq [\gamma, 0, 1]$.
\end{defn}

\begin{prop}\label{prop:tangent-structure-1}
  The functor $T$ shares the following properties with the tangent functor on
  manifolds:
  \begin{itemize}
    \item $T$ preserves finite products: the natural map
      $T(M \times N) \to TM \times TN$ is an isomorphism
      {\cite[Corollary~2.2.19]{B24}};
    \item $T$ preserves subductions {\cite[Proposition~2.2.16]{B24}};
    \item $T$ is local: for $U \subseteq M$ D-open, $T$ of the inclusion is an
      induction with D-open image, so $TU$ is an open subspace of $TM$
      {\cite[Proposition~3.35]{B23}};
    \item $\delta_M$ is a subduction {\cite[Proposition~2.2.20]{B24}}.
  \end{itemize}
\end{prop}

Two further pieces of structure are needed for a tangent category: the
fibrewise addition, whose Euclidean model $\hat{+}$ is defined on the fibre
product $\hat{T}_2 = \hat{T} \times_{\hat\pi} \hat{T}$, and the canonical flip
$\hat{\tau}(u,v_0,v_1,v_{01}) = (u,v_1,v_0,v_{01})$ on the second tangent
functor $\hat{T}^2$, together with the vertical lift
$\hat{\lambda}(u,v) = (u,0,0,v)$. As announced above, the Kan extension does
\emph{not} interact well with these constructions: writing
$T_k \coloneqq T \times_\pi \cdots \times_\pi T$ for the $k$-fold fibre
product, the spaces $\mathbb{L}\hat{T}_k M$ and $T_k M$, and likewise
$\mathbb{L}(\hat{T}^2)M$ and $T^2M$, differ in general, so
$\mathbb{L}\hat{+}$ and $\mathbb{L}\hat{\tau}$ do not immediately descend to
the intended domains. What always exists are the comparison morphisms
\begin{alignat*}{2}
  \theta_k&\colon \mathbb{L}\hat{T}_k \to T_k , \quad
    & \theta_{k,M}[P, u, v_0, \ldots, v_{k-1}]
      &= \bigl([P, u, v_0], \ldots, [P, u, v_{k-1}]\bigr), \\
  \theta^2&\colon \mathbb{L}\hat{T}^2 \to T^2 , \quad
    & \theta^2_M[P, u, v_0, v_1, v_{01}] &= T^2P\,(u, v_0, v_1, v_{01}), \\
  \nu_k&\colon TT_k \to T^2 \times_T \cdots \times_T T^2 , \quad
    & \nu_{k,M}\bigl[(P_1,\ldots,P_k), w\bigr]
      &= \bigl([P_1, w], \ldots, [P_k, w]\bigr).
\end{alignat*}

\begin{ex}\label{ex:non-elastic}
The comparison maps genuinely fail to be isomorphisms: for the quotient
$M = \R/\{\pm 1\}$ of the real line by the reflection $t \mapsto -t$, the map
$\theta_{2,M}$ is not injective, so that the fibrewise addition
$\mathbb{L}\hat{+}$ does not descend to $T_2M$ \cite[Example~3.11]{M}. This is
the phenomenon that the elasticity axioms exclude.
\end{ex}

Blohmann's elasticity axioms single out exactly the spaces on which these
comparison maps behave as they do for manifolds.

\begin{defn}[{\cite[Definition~4.1]{B23}}]\label{def:tangent-structure-1}
  A diffeological space $M$ is \textbf{elastic} if:
  \begin{itemize}
    \item[(E1)] $\theta_{k,M}$ is an isomorphism for every $k \geq 1$;
    \item[(E2)] there is a morphism $\tau_M \colon T^2M \to T^2M$ (the
      \textbf{canonical flip}) with
      $\tau_M \circ \theta^2_M = \theta^2_M \circ (\mathbb{L}\hat{\tau})_M$;
    \item[(E3)] the \textbf{vertical lift}
      $\lambda_M \coloneqq \theta^2_M \circ (\mathbb{L}\hat{\lambda})_M
      \colon TM \to T^2M$ is an induction;
    \item[(E4)] $\nu_{k,M}$ is injective for every $k \geq 1$;
    \item[(E5)] each iterated space $T_{k_1}\cdots T_{k_n}M$ satisfies
      (E1)--(E4).
  \end{itemize}
  Elastic spaces span a full subcategory
  $\mathsf{Elast} \subseteq \mathsf{Diff}$.
\end{defn}

\begin{thm}[{\cite[Theorem~4.2]{B23}}]\label{thm:tangent-structure-2}
  The functor $T$, the ring $\R$, and the natural transformations $\pi$, $0$,
  $\kappa$, $+ \coloneqq \mathbb{L}\hat{+} \circ \theta_2^{-1}$, the flip
  $\tau$ of \textup{(E2)}, and the lift $\lambda$ of \textup{(E3)} make
  $\mathsf{Elast}$ a Cartesian tangent category with scalar
  $\R$-multiplication.
\end{thm}

\begin{rem}\label{rem:negatives}
In $\mathsf{Elast}$ the negatives required in
Definition~\ref{def:tangent-structure-axioms} are supplied explicitly by the
scalar $\R$-multiplication: for
$v \in TM$ one has
$v + \kappa_M(-1, v) = 0_M(\pi_M(v))$, as one checks on representatives using
\eqref{eq:TP-linear} below. Hence every $\pi_M \colon TM \to M$, and likewise
every $\pi_{TM} \colon T^2M \to TM$, is a bundle of abelian \emph{groups}, and
differences of morphisms with the same base point are defined. We use this
silently from Section~\ref{sec:connections} on.
\end{rem}

The class of elastic spaces is sufficiently broad for our purposes: it contains
all finite-dimensional manifolds, all convenient vector spaces---in particular
all Banach and Fr\'echet spaces---and the manifolds modelled on them
\cite[Section~2.1]{M}, and it is stable under the mapping-space constructions
of Section~\ref{sec:mapping-space} \cite{B23}. Let us emphasize at once that
this breadth has a price. Smooth dynamical systems on a convenient vector space
need not have unique solutions, so $\mathsf{Elast}$ is strictly larger than the
subcategory $\mathsf{Elast}_{\mathrm{curve}}$, whose uniqueness property is
used in Section~\ref{sec:geodesics} to obtain uniquely determined geodesics;
the two should be kept carefully apart. From now on, unless stated otherwise, every diffeological space is
assumed elastic, and we freely use the structure maps of
Theorem~\ref{thm:tangent-structure-2}.

\begin{rem}\label{rem:cocone-TP}
Since $\mathcal{Y}$ is full and faithful, the left Kan extension along
$\mathcal{Y}$ restricts on representables to the original functor; in
particular $T(\mathcal{Y}U) \cong \mathcal{Y}(\hat{T}U)$ for every
$U \in \mathsf{Euc}$. Under this identification the cocone map of the colimit
$TM = \colim_{P \in \D(M)} \hat{T}U_P$,
\begin{equation*}
  \hat{T}U_P = U_P \times \R^{\dim U_P} \longrightarrow TM,
  \qquad (u,v) \longmapsto [P,u,v],
\end{equation*}
is precisely the tangent map $T(P)$ of the plot $P$: applying the formula of
Remark~\ref{rem:tangent-structure-2} to the identity plot of $U_P$ gives
$T(P)\,[\id_{U_P},u,v] = [P,u,v]$. We henceforth write $T(P)$ for the cocone
maps.
\end{rem}

\begin{lem}\label{lem:TP-linear}
Let $M$ be elastic. Each cocone map $T(P) \colon \hat{T}U_P \to TM$ is a
linear bundle morphism (Definition~\ref{def:differential-bundle-morphism})
over $P$; explicitly, on representatives,
\begin{equation}\label{eq:TP-linear}
  \kappa_M\bigl(r,[P,u,v]\bigr) = [P,u,rv],
  \qquad
  [P,u,v_0] +_M [P,u,v_1] = [P,u,v_0+v_1].
\end{equation}
Moreover the family $\{T(P)\}_{P \in \D(M)}$ is jointly epic, and so is the
family $\{T^{(k)}(P)\}$ of cocone maps of
$\mathbb{L}\hat{T}_kM = \colim_P \hat{T}_kU_P$.
\end{lem}

\begin{proof}
Compatibility with $\pi$ and $0$ is immediate from
Definition~\ref{def:nat-trans}. The first identity in \eqref{eq:TP-linear} is
the naturality of $\kappa$ applied to $P \colon U_P \to M$, and the second is
the naturality of $+$, which is available because $M$ is elastic and hence
$\theta_{2,M}$ is invertible by (E1). Joint epimorphy holds because in
$\mathsf{Diff}$ a colimit is a quotient of a coproduct, so the cocone maps are
jointly surjective and the colimit diffeology is the pushforward of the
diffeologies of the $\hat{T}U_P$.
\end{proof}

\subsection{Vector fields, one-forms and the Lie bracket}

This subsection follows the general framework developed in
\cite[Sections~2.4, 2.6]{AB} and \cite[Section~A.1]{M}, where the statements
are established for an arbitrary Cartesian tangent category with scalar
$\R$-multiplication and therefore apply verbatim to $\mathsf{Elast}$ by
Theorem~\ref{thm:tangent-structure-2}. In contrast with
Section~\ref{subsec:elastic}, the results below that are specific to
the diffeological setting are proved in full, since they are used repeatedly in
the sequel.

\begin{defn}[Vector field]\label{def:vector-field}
Let $M \in \mathsf{Elast}$.
A \textbf{vector field} $V \in \mathfrak{X}(M)$ on $M$ is a smooth section
$V \colon M \to T(M)$, that is, a smooth map satisfying $\pi_M \circ V = \id_M$.
\end{defn}

\begin{defn}\label{def:vf-module}
For $V, W \in \mathfrak{X}(M)$ and $f \in C^\infty(M,\R)$ we set
\begin{equation*}
  V + W \coloneqq {+}_M \circ \langle V, W\rangle,
  \qquad
  fV \coloneqq \kappa_M \circ \langle f, V \rangle ,
\end{equation*}
where $\langle -,- \rangle$ denotes the induced morphism into the fibre
product, respectively into $\R \times TM$. Both are again smooth sections,
since $+_M$ and $\kappa_M$ are morphisms of $\mathsf{Elast}$ over $M$.
\end{defn}

\begin{lem}\label{lem:vf-module}
With the operations of Definition~\ref{def:vf-module}, $\mathfrak{X}(M)$ is a
module over the commutative $\R$-algebra $C^\infty(M,\R)$; in particular it is
an $\R$-module, the scalar $r \in \R$ acting through the constant function $r$.
\end{lem}

\begin{proof}
All axioms are equalities of morphisms $M \to TM$, and by
Lemma~\ref{lem:TP-linear} they may be verified on representatives $[P,u,v]$,
where they reduce to the corresponding identities in the fibres
$\R^{\dim U_P}$ of $\hat{T}U_P$. Negatives are supplied by
Remark~\ref{rem:negatives}.
\end{proof}

We shall use the $C^\infty(M,\R)$-module structure of
Lemma~\ref{lem:vf-module} throughout Sections~\ref{sec:connections}
and~\ref{sec:riemannian-metric}, where the tensoriality of connections and
metrics is expressed in terms of it.

Before defining the action of vector fields on functions we identify the
fibrewise-linear functions on $TM$ with the $1$-forms of Iglesias-Zemmour.

\begin{lem}\label{lem:cotangent-omega}
Let $M$ be an elastic diffeological space, and write $\Clin(-,\R)$ for the set
of fibrewise-linear smooth maps. Then the assignment
$\phi\mapsto(\phi\circ T(P))_P$ is a bijection
\[
  \Clin(TM,\R)
  \;\xrightarrow{\ \sim\ }\;
  \lim_{P\in\D(M)}\Clin(\hat{T}U_P,\R)
  =\lim_{P\in\D(M)}\Omega^{1}(U_P)
  =\Omega^{1}(M),
\]
the limits being taken in the category of sets and the last equality being the
definition of a differential $1$-form on a diffeological space (Section~\ref{subsec:diffeology}).
\end{lem}

\begin{proof}
Since $C^{\infty}(-,\R)$ is contravariant and carries the colimit
$TM=\colim_P\hat{T}U_P$ to a limit,
\begin{equation}\label{eq:nonlinear-iso}
  C^{\infty}(TM,\R)\;\xrightarrow{\ \sim\ }\;
  \lim_{P}C^{\infty}(\hat{T}U_P,\R),
  \qquad
  \phi\longmapsto(\phi\circ T(P))_{P},
\end{equation}
is a bijection, so it suffices to show that $\phi$ is fibrewise linear if and
only if $\phi_P:=\phi\circ T(P)$ is fibrewise linear for every $P$.

If $\phi$ is fibrewise linear, then so is each $\phi_P$, since $T(P)$ is a
linear bundle morphism (Lemma~\ref{lem:TP-linear}).

Conversely, assume every $\phi_P$ is fibrewise linear. For additivity we must
show $\phi\circ+_M=+_{\R}\circ(\phi\times_M\phi)$ as maps $T_2M\to\R$. Since
$+_M=\mathbb{L}\hat{+}\circ\theta_{2,M}^{-1}$ with
$\theta_{2,M}\colon\mathbb{L}\hat{T}_2M\xrightarrow{\sim}T_2M$ the elastic
isomorphism~(E1), this is equivalent to
$\phi\circ\mathbb{L}\hat{+}=+_{\R}\circ(\phi\times_M\phi)\circ\theta_{2,M}$
on $\mathbb{L}\hat{T}_2M=\colim_P\hat{T}_2U_P$. The cocone maps of this colimit
are jointly epic (Lemma~\ref{lem:TP-linear}), so it suffices to check the
equality on representatives:
\[
  \phi\circ\mathbb{L}\hat{+}\,[P,u,v_0,v_1]
  =\phi_P(u,v_0+v_1)
  =\phi_P(u,v_0)+\phi_P(u,v_1)
  =+_{\R}\circ(\phi\times_M\phi)\,\theta_{2,M}[P,u,v_0,v_1].
\]
Homogeneity is verified in the same manner, on representatives of $TM$
directly: by \eqref{eq:TP-linear} and homogeneity of $\phi_P$,
\[
  \phi\bigl(\kappa_M(r,[P,u,v])\bigr)
  =\phi_P(u,rv)
  =r\,\phi_P(u,v)
  =r\,\phi[P,u,v].
\]
Thus \eqref{eq:nonlinear-iso} restricts to the asserted bijection.
\end{proof}

The categorical definition of the action of a vector field on a function uses
only the tangent structure; Proposition~\ref{prop:pr2-Tf} below identifies it
with the classical formula.

\begin{defn}\label{def:vf-action}
For a function $f \colon M \to \R$ and a vector field $V \in \mathfrak{X}(M)$,
the \textbf{action} of $V$ on $f$ is the function $V \cdot f \colon M \to \R$
defined by
\begin{align*}
    V \cdot f := \pr_2 \circ T(f) \circ V ,
\end{align*}
where $\pr_2 \colon T\R = \R \times \R \to \R$ is the second projection.
\end{defn}

\begin{prop}\label{prop:pr2-Tf}
  Let $d \colon \Omega^n(M) \to \Omega^{n+1}(M)$ denote the exterior derivative.
  For any smooth function $f \colon M \to \R$ one has
  $\pr_2 \circ T(f) = d(f)$
  under the identification of Lemma~\ref{lem:cotangent-omega}; consequently
  $V \cdot f = d(f) \circ V$ for every $V \in \mathfrak{X}(M)$.
\end{prop}

\begin{proof}
  By definition, $d(f) \in \Omega^1(M) = \lim_P \Omega^1(U_P)$ is the family
  $(d(f \circ P))_{P \in \D(M)}$, and for a function $g$ on a domain one has
  $d(g) = \pr_2 \circ \hat{T}g$. Applying
  Remark~\ref{rem:tangent-structure-2} and $T\R = \hat{T}\R$,
  \begin{equation*}
    \bigl(\pr_2 \circ T(f)\bigr)\circ T(P)\,(u,v)
      = \pr_2\bigl[f \circ P, u, v\bigr]
      = \pr_2 \circ \hat{T}(f \circ P)(u,v)
      = d(f \circ P)_u(v)
  \end{equation*}
  for every plot $P$. Thus the image of $\pr_2 \circ T(f) \in \Clin(TM,\R)$
  under the bijection of Lemma~\ref{lem:cotangent-omega} is the family
  $(d(f\circ P))_P = d(f)$, and injectivity of that bijection gives
  $\pr_2 \circ T(f) = d(f)$.
\end{proof}

The action of vector fields on functions has the following properties.

\begin{prop}[{\cite[Proposition~2.25]{AB}}]\label{prop:vf-action-properties}
  For any $V, W \in \mathfrak{X}(M)$ and $f, g \in C^\infty(M, \R)$,
  the following hold:
  \begin{enumerate}
    \item $V \cdot (f + g) = V \cdot f + V \cdot g$;
    \item $(V + W) \cdot f = V \cdot f + W \cdot f$;
    \item $V \cdot (fg) = (V \cdot f)g + f(V \cdot g)$.
  \end{enumerate}
\end{prop}

We now turn to the Lie bracket, for which we first record the standard
extension of the vertical lift.

\begin{defn}\label{def:vertical-lift-extension}
Using the addition of the tangent structure as a bundle of abelian groups
(Remark~\ref{rem:negatives}), the \textbf{extended vertical lift} is the
morphism $\lambda_2 \colon T_2 \to T^2$ defined by
\begin{align*}
    \lambda_2 := \tau \circ {+}_T \circ (T0 \times_0 \lambda),
\end{align*}
where $T0 \times_0 \lambda$ denotes the map induced on the fibre product
$T_2 = T\times_\pi T$ by $T0$ on the first factor and $\lambda$ on the second;
component-wise,
$\lambda_{2,M} = \tau_M \circ {+}_{TM} \circ (T0_M \times_{0_M} \lambda_M)$, so
that informally $\lambda_{2,M}(v,w) = \tau_M(T0_M(v)) +_{TM} \lambda_M(w)$.
\end{defn}

\begin{prop}[{\cite[Section~3.4]{CC14}}]\label{prop:lambda2-iso}
For every $M \in \mathsf{Elast}$ the morphism $\lambda_{2,M}$ is an
isomorphism from $T_2(M) = T(M) \times_M T(M)$ onto the vertical tangent bundle
$\Ker(T\pi_M) \subseteq T^2M$.
\end{prop}

\begin{proof}
This is the universality of the vertical lift, one of the axioms of a tangent
category (Definition~\ref{def:tangent-structure-axioms}), combined with
Theorem~\ref{thm:tangent-structure-2}.
\end{proof}

As in the case of manifolds, one can now define a Lie bracket on vector fields
\cite[Section~3.4]{CC14}. Let $V, W$ be vector fields on $M$. By naturality of
$\pi$ and the axiom $\pi_{TM} \circ \tau_M = T(\pi_M)$, both
$T(W) \circ V$ and $\tau_M \circ T(V) \circ W$ lie over $W$ with respect to
$\pi_{TM}$, so their difference
\begin{equation*}
  \ell(V,W) = T(W) \circ V - \tau_M \circ T(V) \circ W \colon M \to T^2M
\end{equation*}
is defined in the bundle of abelian groups $\pi_{TM}\colon T^2M \to TM$
(Remark~\ref{rem:negatives}). Since $T(\pi_M) \circ \ell(V,W) = V - V = 0_M$,
it takes values in $\Ker(T\pi_M)$, which is the image of $\lambda_{2,M}$
(Proposition~\ref{prop:lambda2-iso}); as $\pi_{TM} \circ \ell(V,W) = W$, it
factors uniquely through $\lambda_{2,M}$ with first component $W$.

\begin{defn}\label{def:lie-bracket}
The \textbf{Lie bracket} of vector fields $V, W \in \mathfrak{X}(M)$
is defined as the unique vector field $[V,W] \in \mathfrak{X}(M)$ satisfying
\begin{equation*}
  \ell(V,W) = \lambda_{2,M} \circ \langle W, [V,W] \rangle,
\end{equation*}
where $\langle W, [V,W] \rangle \colon M \to T(M) \times_M T(M)$
is the induced morphism into the fibre product.
\end{defn}

Rosick\'{y} \cite{R} announced that the bracket defined above satisfies
the Jacobi identity; Cockett and Cruttwell \cite{CC15} provided a complete
proof.

\begin{prop}[{\cite[Theorem~4.2]{CC15}}]\label{thm:jacobi}
In a tangent category with negatives, the bracket of
Definition~\ref{def:lie-bracket} is antisymmetric and satisfies the Jacobi
identity for every object $M$.
\end{prop}

In the presence of scalar $\R$-multiplication the bracket is moreover
$\R$-bilinear (see \cite[Section~2.6]{AB}).

Consequently, for $M \in \mathsf{Elast}$, $\mathfrak{X}(M)$ is an $\R$-Lie algebra. The bracket is
moreover compatible with the action on functions and with the
$C^\infty(M,\R)$-module structure of Lemma~\ref{lem:vf-module}.

\begin{prop}[{\cite[Propositions~2.25, 2.27]{AB}}]\label{prop:leibniz-bracket}
For any $f \colon M \to \R$ and $V, W \in \mathfrak{X}(M)$,
\begin{align*}
    [V, W] \cdot f = V \cdot (W \cdot f) - W \cdot (V \cdot f),
    \qquad
    [V, fW] = (V \cdot f)W + f[V, W].
\end{align*}
\end{prop}

\section{Connections and covariant derivatives}\label{sec:connections}

Throughout this section, $M$ denotes an elastic diffeological space, and we
freely use the tangent structure of Theorem~\ref{thm:tangent-structure-2}.
The notions of this section are those of Cockett and Cruttwell~\cite{CC17} and
Lucyshyn-Wright~\cite{LW} for a general tangent category, specialized to
$\mathsf{Elast}$, with two additional axioms for scalar multiplication
(Remark~\ref{rem:scalar-axiom}). Results taken from these sources are stated
with a reference and without proof; proofs are included where the scalar axioms
enter.

\subsection{Vertical connections and covariant derivatives}%
\label{subsec:vertical-cov}

Connections in tangent categories are described by morphisms on the second tangent bundle rather than directly by covariant derivatives. This viewpoint is particularly suitable in the diffeological setting, where tangent spaces need not admit local trivializations and many classical constructions are unavailable. We begin with the notion of a vertical connection and then derive the corresponding covariant derivative.

\begin{defn}[{\cite[Definitions~3.2, 3.3]{CC17}, modified}]\label{def:vertical-connection}
  A morphism $K \colon T^2M \to TM$ satisfying the following conditions
  is called a \textbf{vertical (affine) connection} on $M$:
  \begin{enumerate}[(a)]
    \item $K \circ \lambda_M = \id_{TM}$
    \item $\pi_M \circ K = \pi_M \circ \pi_{TM}$
    \item $\lambda_M \circ K = T(K) \circ \lambda_{TM}$
    \item $\lambda_M \circ K = T(K) \circ \tau_{TM} \circ T(\lambda_M)$
    \item $K \circ \kappa_{TM} = \kappa_M \circ (\id_{\R} \times K)$
    \item $K \circ T_{(2)}\kappa_M = \kappa_M \circ (\id_{\R} \times K)$
  \end{enumerate}
\end{defn}

\begin{rem}\label{rem:K-linear-bundle}
  By definition, the vertical connection $K$ yields linear bundle morphisms
  (Definition~\ref{def:differential-bundle-morphism}):
  \[
    (K, \pi_M) \colon (T(\pi_M) \colon T^2M \to TM) \to (\pi_M \colon TM \to M),
  \]
  \[
    (K, \pi_M) \colon (\pi_{TM} \colon T^2M \to TM) \to (\pi_M \colon TM \to M).
  \]
  Consequently, $K$ is additive.
\end{rem}

\begin{rem}\label{rem:scalar-axiom}
  In general tangent structures as considered in \cite{CC14, CC15, CC17},
  scalar multiplication on bundles is not part of the axioms.
  In elastic diffeology, however, the structure of scalar $\R$-multiplication
  is present. Taking into account compatibility with connections on tangent spaces
  of manifolds, we have added preservation of scalar multiplication to the axioms
  for both vertical and horizontal connections: axioms~(e) and~(f) for vertical connections,
  and axioms~(e) and~(f) of Definition~\ref{def:horizontal-connection} for
  horizontal connections below. These axioms ensure that the induced covariant derivative
  is linear with respect to scalar multiplication and satisfies the classical Leibniz identities.
\end{rem}

\begin{rem}\label{rem:K-subduction}
  Since $K$ is a split epimorphism, it is a strong epimorphism.
  Strong epimorphisms in the category $\mathsf{Diff}$ of diffeological spaces
  are subductions, so $K$ is a subduction.
\end{rem}

\begin{ex}[Vertical connections on Euclidean space]\label{ex:euclidean-vertical-connection}
Let $M = \R^n$, identified with $TM = \R^n \times \R^n$ and
$T^2M = \R^n \times \R^n \times \R^n \times \R^n$, with coordinates
$(u, v_0, v_1, v_{01}) \in T^2\R^n$ so that
$\pi_{TM}(u, v_0, v_1, v_{01}) = (u, v_0)$ and
$T(\pi_M)(u, v_0, v_1, v_{01}) = (u, v_1)$.
For a smooth family $\Gamma(u)$ of bilinear maps
$\R^n \times \R^n \to \R^n$, the morphism
\begin{align*}
    K_\Gamma(u, v_0, v_1, v_{01}) = \bigl(u,\; v_{01} + \Gamma(u)(v_0, v_1)\bigr)
\end{align*}
is a vertical connection (see \cite[Example~3.6]{CC17}). The bilinearity of
$\Gamma(u)$ is essential: the scalar axioms (e) and (f) amount to the
homogeneity of $\Gamma(u)$ in each of its two arguments.
The special case $\Gamma = 0$, namely
$K_0(u, v_0, v_1, v_{01}) = (u, v_{01})$, is called the
\emph{canonical affine vertical connection}.
Moreover, since $\tau_M(u, v_0, v_1, v_{01}) = (u, v_1, v_0, v_{01})$, we have
$K_\Gamma \circ \tau_M = K_\Gamma$ if and only if each $\Gamma(u)$ is
symmetric; thus torsion-freeness (Definition~\ref{def:torsion-free} below)
recovers the classical symmetry of the Christoffel symbols.
\end{ex}

\begin{defn}[{\cite[Definition~3.15]{CC17}}]\label{def:covariant-derivative}
  Let $K$ be a vertical connection on $M$ and let $X, Y \in \mathfrak{X}(M)$.
  The operation
  \begin{align*}
    \nabla^K \colon \mathfrak{X}(M) \times \mathfrak{X}(M) \to \mathfrak{X}(M),\qquad
    \nabla^K_X Y := K \circ T(Y) \circ X
  \end{align*}
  is called the \textbf{covariant derivative for $K$}.
\end{defn}

\begin{rem}\label{rem:cov-deriv-vs-connection}
  In ordinary Riemannian geometry, connections and covariant derivatives
  determine each other. In a tangent category this need not be the case:
  in the absence of local triviality, the family of morphisms
  $T(Y) \circ X$, for varying $X, Y \in \mathfrak{X}(M)$, need not be
  jointly epic on $T^2M$, so the covariant derivative $\nabla^K$ need not
  determine the vertical connection $K$.
\end{rem}

\begin{defn}[{\cite[Definition~3.24]{CC17}}]\label{def:torsion-free}
  A vertical connection $K$ on $M$ is said to be \textbf{torsion-free} if
  \begin{align*}
    K \circ \tau_M = K.
  \end{align*}
\end{defn}

\begin{prop}[{\cite[Corollary~3.32]{CC17}}]\label{prop:torsion-free-bracket}
  If a vertical connection $K$ on $M$ is torsion-free, then for all
  $X, Y \in \mathfrak{X}(M)$,
  \begin{align*}
    \nabla^K_X Y - \nabla^K_Y X = [X, Y].
  \end{align*}
\end{prop}

\begin{prop}\label{prop:covariant-derivative-properties}
  Let $K$ be a vertical connection on $M$,
  let $X, Y, Z \in \mathfrak{X}(M)$, and let $f \colon M \to \R$ be a smooth
  function. The covariant derivative $\nabla^K$ satisfies the following:
  \begin{enumerate}
    \item $\nabla^K_{X+Y}Z = \nabla^K_X Z + \nabla^K_Y Z$
    \item $\nabla^K_{fX}Y = f\nabla^K_X Y$
    \item $\nabla^K_X (Y + Z) = \nabla^K_X Y + \nabla^K_X Z$
    \item $\nabla^K_X (fY) = (Xf)Y + f\nabla^K_X Y$.
  \end{enumerate}
\end{prop}

\begin{proof}
  \begin{enumerate}
    \item Since $K$ and $T(Z)$ are additive,
    \begin{align*}
      \nabla^K_{X+Y} Z
        = K \circ T(Z) \circ (X + Y)
        = K \circ T(Z) \circ X + K \circ T(Z) \circ Y
        = \nabla^K_X Z + \nabla^K_Y Z.
    \end{align*}
    \item By naturality of the scalar multiplication $\kappa$ and
      the (additional) axiom (e) for vertical connections,
    \begin{align*}
      \nabla^K_{fX}Y
        &= K \circ T(Y) \circ \kappa_M(f, X)
         = K \circ \kappa_{TM}(f, T(Y)\circ X) \\
        &= \kappa_M(f, K \circ T(Y) \circ X)
         = f\nabla^K_X Y.
    \end{align*}
    \item Since the tangent functor preserves addition and $K$ is additive,
    \begin{align*}
      \nabla^K_X (Y + Z)
        = K \circ T(Y+Z) \circ X
        = K \circ T(Y) \circ X + K \circ T(Z) \circ X
        = \nabla^K_X Y + \nabla^K_X Z.
    \end{align*}
    \item By the definition of the covariant derivative,
      $\nabla^K_X (fY) = K \circ T(\kappa_M \circ (f, Y)) \circ X$.
      By functoriality of $T$ and
      Proposition~\ref{prop:decomposition-tangent-morphism},
    \begin{align*}
     & K \circ T(\kappa_M \circ (f, Y)) \circ X \\
        &= K \circ \bigl(T_{(1)} \kappa_M \circ (\id_{T\R} \times \pi_{TM})
            \circ \chi_{\R,TM}
           + T_{(2)} \kappa_M \circ (\pi_{\R} \times \id_{TM})
            \circ \chi_{\R,TM}\bigr)
          \circ (T(f)\circ X, T(Y) \circ X) \\
        &= K \circ T_{(1)} \kappa_M \circ (T(f)\circ X, \pi_{TM} \circ T(Y) \circ X)
           + K \circ T_{(2)} \kappa_M \circ (\pi_{\R} \circ T(f)\circ X, T(Y) \circ X) \\
        &= K \circ T_{(1)} \kappa_M \circ (T(f)\circ X, Y \circ \pi_M \circ X)
           + K \circ T_{(2)} \kappa_M \circ (f \circ \pi_M \circ X, T(Y) \circ X) \\
        &= K \circ T_{(1)} \kappa_M \circ (T(f)\circ X, Y)
           + K \circ T_{(2)} \kappa_M \circ (f, T(Y) \circ X).
    \end{align*}
    We examine each term.
    For $K \circ T_{(1)} \kappa_M \circ (T(f)\circ X, Y)$,
    the second diagram in Definition~\ref{def:scalar-multiplication} gives
    \begin{align*}
      & K \circ T_{(1)} \kappa_M \circ (T(f)\circ X, Y) \\
        &= K \circ \lambda_{2, M} \circ
          \bigl(\kappa_M \circ (\pr_1, \pr_3),\; \kappa_M \circ (\pr_2, \pr_3)\bigr)
          \circ (T(f)\circ X, Y) \\
        &= K \circ \tau_M \circ {+}_{TM} \circ (T0_M \times_{0_M} \lambda_M)
          \circ \bigl(\kappa_M \circ (\pr_1 \circ T(f) \circ X, Y),\;
                 \kappa_M \circ (\pr_2 \circ T(f) \circ X, Y)\bigr) \\
        &= K \circ \tau_M \circ \bigl(0 + \lambda_M \circ \kappa_M
           \circ (\pr_2 \circ T(f) \circ X, Y)\bigr) \\
        &= K \circ \tau_M \circ \lambda_M \circ \kappa_M \circ (df \circ X, Y),
    \end{align*}
    where the last equality uses $\pr_2 \circ T(f) = d(f)$
    (Proposition~\ref{prop:pr2-Tf}).
    Using $\tau_M \circ \lambda_M = \lambda_M$
    (Section~\ref{sec:RosickysAxioms}) and axiom (a),
    we obtain
    \begin{align*}
      K \circ T_{(1)} \kappa_M \circ (T(f)\circ X, Y)
        = \kappa_M \circ (df \circ X, Y) = (Xf)Y,
    \end{align*}
    the last equality by Definition~\ref{def:vf-action}.
    For $K \circ T_{(2)} \kappa_M \circ (f, T(Y) \circ X)$,
    the axiom (f) in Definition~\ref{def:vertical-connection} gives
    \begin{align*}
     K \circ T_{(2)} \kappa_M \circ (f, T(Y) \circ X)
       &= \kappa_{M} \circ  (\id_\R \times K) \circ (f, T(Y) \circ X)\\
       &= \kappa_{M} \circ (f, K \circ T(Y) \circ X) =  f\nabla^K_X Y.
    \end{align*}
    Therefore $\nabla^K_X (fY) = (Xf)Y + f\nabla^K_X Y$.
  \end{enumerate}
\end{proof}

\begin{ex}\label{ex:covariant-derivative}
 Let $M = \R^n$ with the vertical connection $K_\Gamma$ of
Example~\ref{ex:euclidean-vertical-connection}. Write vector fields as
$X(x) = (x, \bar{X}(x))$ and $Y(x) = (x, \bar{Y}(x))$ with
$\bar{X}, \bar{Y} \colon \R^n \to \R^n$. The tangent map of $Y$ is
$T(Y)(x, w) = \bigl(x, \bar{Y}(x), w, d\bar{Y}_x(w)\bigr)$, so
\begin{align*}
  (T(Y) \circ X)(x)
    = \bigl(x,\; \bar{Y}(x),\; \bar{X}(x),\; d\bar{Y}_x(\bar{X}(x))\bigr),
\end{align*}
and applying $K_\Gamma$ yields
\begin{align*}
  (\nabla^{K_\Gamma}_X Y)(x)
    = \bigl(x,\; d\bar{Y}_x(\bar{X}(x))
      + \Gamma(x)(\bar{Y}(x), \bar{X}(x))\bigr).
\end{align*}
For the canonical connection $K_0$ this is the directional derivative
$\nabla^{K_0}_X Y = d\bar{Y}(\bar{X})$, and in general it is the classical
local expression of a covariant derivative through Christoffel symbols.
Note the order of the arguments of $\Gamma(x)$: the second component of
$T(Y) \circ X$ (over $\pi_{TM}$) carries the value of $Y$ and the third
(over $T(\pi_M)$) that of $X$; when $\Gamma(x)$ is symmetric, i.e.\ when
$K_\Gamma$ is torsion-free, the order is immaterial.
\end{ex}

\subsection{Horizontal, full and effective connections}%
\label{subsec:full-connection}

\begin{defn}[{\cite[Definitions~4.5, 4.6]{CC17}, modified}]\label{def:horizontal-connection}
An \textbf{affine horizontal connection} on $M$ is a
smooth morphism $H \colon TM \times_M TM \to T^2M$ satisfying the following conditions
with respect to the projections $\pr_1, \pr_2 \colon TM \times_M TM \to TM$.
We denote the scalar multiplications on the fibre product $TM \times_M TM$
induced by $\kappa_M$ on each factor by
\begin{align*}
  \kappa^{(1)}, \kappa^{(2)} \colon \R \times (TM \times_M TM) \to TM \times_M TM,
\end{align*}
that is, $\kappa^{(1)}(r, (v_0, v_1)) = (\kappa_M(r, v_0), v_1)$ and
$\kappa^{(2)}(r, (v_0, v_1)) = (v_0, \kappa_M(r, v_1))$.
\begin{enumerate}[(a)]
    \item $T(\pi_M) \circ H = \pr_1$
    \item $\pi_{TM} \circ H = \pr_2$
    \item $\lambda_{TM} \circ H = T(H) \circ (\lambda_M \times 0_{TM})$
    \item $\tau_{TM} \circ T(\lambda_M) \circ H = T(H) \circ (0_{TM} \times \lambda_M)$
    \item $H \circ \kappa^{(1)} = \kappa_{TM} \circ (\id_{\R} \times H)$
    \item $H \circ \kappa^{(2)} = T_{(2)}\kappa_M \circ (\id_{\R} \times H)$
\end{enumerate}
\end{defn}

\begin{rem}\label{rem:horizontal-scalar-axiom}
  Axioms~(e) and~(f) express the preservation of scalar multiplication with respect
  to each of the two tangent bundle structures on $T^2M$:
  $\pi_{TM} \colon T^2M \to TM$ (with scalar multiplication $\kappa_{TM}$) and
  $T(\pi_M) \colon T^2M \to TM$ (with scalar multiplication $T_{(2)}\kappa_M$).
\end{rem}

\begin{ex}[{\cite[Example~3.6]{LW}}]%
\label{ex:euclidean-horizontal-connection}
In the coordinates of Example~\ref{ex:euclidean-vertical-connection}, the
\emph{canonical horizontal connection}
$H \colon T\R^n \times_{\R^n} T\R^n \to T^2\R^n$ is given by
\begin{align*}
    H\bigl((u, v_0),\, (u, v_1)\bigr) = (u, v_1, v_0, 0).
\end{align*}

One verifies both scalar axioms directly: $\kappa_{TM}$ scales the $(v_1, v_{01})$ components, which carry the first argument by axiom~(a), while $T_{(2)}\kappa_M$ scales the $(v_0, v_{01})$ components, which carry the second argument by axiom~(b); the last component of $H$ vanishes.
\end{ex}

\begin{defn}[{\cite[Definition~5.2]{CC17}}]\label{def:full-connection}
A pair $(K, H)$ of a vertical connection $K \colon T^2M \to TM$ and a horizontal
connection $H \colon TM \times_M TM \to T^2M$ on $M$
is called a \textbf{full connection} if the
following two conditions hold:
\begin{enumerate}[(a)]
    \item $K \circ H = 0_M \circ \pi_M \circ \pr_1$,
    \item $\mu \circ \langle K, \pi_{TM} \rangle + H \circ U = \id_{T^2M}$.
\end{enumerate}
Here
\begin{align*}
  \mu &\coloneqq T({+}) \circ ((\lambda \circ \pr_1) \times (0 \circ \pr_2))
    \colon TM \times_M TM \to T^2M, \\
  U &\coloneqq \langle T(\pi_M), \pi_{TM} \rangle
    \colon T^2M \to TM \times_M TM;
\end{align*}
condition (b) states that every element of $T^2M$ is the sum of its
vertical and horizontal parts.
\end{defn}

\begin{prop}[{\cite[Proposition~5.8]{CC17}}]\label{prop:full-connection-fibre-product}
  Let $(K, H)$ be a full connection on $M$.
  Then the following is a fibre product diagram:
  \begin{equation}\label{eq:cx_ld}
\xymatrix{
            & T^2M \ar[dl]_{\pi_{TM}} \ar[d]|{T(\pi_M)} \ar[dr]^K & \\
TM \ar[dr]_{\pi_M} & TM \ar[d]|{\pi_M}                          & TM \ar[dl]^{\pi_M}\\
   & M                                                 &
}
\end{equation}
  That is,
  \begin{align*}
    T^2M \cong TM \times_M TM \times_M TM.
  \end{align*}
\end{prop}

\begin{rem}\label{rem:theta-bi-scalar}
Recall that $T^2M$ carries two scalar multiplications over $TM$: the
multiplication $\kappa_{TM} \colon \R \times T^2M \to T^2M$ for the bundle
$\pi_{TM}$, and the multiplication $T_{(2)}\kappa_M \colon \R \times T^2M \to T^2M$
for the bundle $T(\pi_M)$, where
\[
  T_{(2)}\kappa_M = \tau_M \circ \kappa_{TM} \circ (\id_{\R} \times \tau_M)
\]
is the partial tangent of $\kappa_M$ (Definition~\ref{def:scalar-multiplication}).
On the fibre product $TM \times_M TM \times_M TM$ we write $\kappa^{(2,3)}$
(resp.\ $\kappa^{(1,3)}$) for the scalar multiplication that scales the second and
third (resp.\ the first and third) factors by $\kappa_M$ and leaves the remaining
factor fixed.

If the vertical connection $K$ satisfies axioms (e) and (f) of
Definition~\ref{def:vertical-connection}, then the isomorphism
$\Theta = \langle \pi_{TM},\, T(\pi_M),\, K \rangle$ of
Proposition~\ref{prop:full-connection-fibre-product} is moreover an isomorphism of
$\R$-modules for \emph{both} structures:
\begin{align*}
  \Theta \circ \kappa_{TM}
    &= \kappa^{(2,3)} \circ (\id_{\R} \times \Theta),
  &
  \Theta \circ T_{(2)}\kappa_M
    &= \kappa^{(1,3)} \circ (\id_{\R} \times \Theta).
\end{align*}
Indeed, the three components are checked separately.
For $\kappa_{TM}$: the projection $\pi_{TM}$ is left fixed (it is the base of
$\kappa_{TM}$), while
\[
  T(\pi_M) \circ \kappa_{TM} = \kappa_M \circ (\id_{\R} \times T(\pi_M))
  \quad\text{(naturality of $\kappa$),}
  \qquad
  K \circ \kappa_{TM} = \kappa_M \circ (\id_{\R} \times K)
  \quad\text{(axiom (e)).}
\]
For $T_{(2)}\kappa_M$: the projection $T(\pi_M)$ is left fixed, while, using
$\pi_{TM} \circ \tau_M = T(\pi_M)$ together with naturality of $\kappa$,
\[
  \pi_{TM} \circ T_{(2)}\kappa_M = \kappa_M \circ (\id_{\R} \times \pi_{TM}),
  \qquad
  K \circ T_{(2)}\kappa_M = \kappa_M \circ (\id_{\R} \times K)
  \quad\text{(axiom (f)).}
\]
In particular both $\kappa_{TM}$ and $T_{(2)}\kappa_M$ scale the third
($K$-)factor, reflecting that the second-order component of $T^2M$ is scaled by
each of the two structures.
\end{rem}

\begin{defn}[{\cite[Definition~6.6, Theorem~8.4]{LW}}]\label{def:effective-vertical-connection}
An affine vertical connection $K \colon T^2M \to TM$ on $M$ is called
\textbf{effective} if the morphism
\begin{align*}
    \Theta \colon T^2M \longrightarrow TM \times_M TM \times_M TM,
    \qquad W \longmapsto \bigl(\pi_{TM}(W),\; T(\pi_M)(W),\; K(W)\bigr)
\end{align*}
is a diffeomorphism of elastic diffeological spaces.
\end{defn}

\begin{prop}[{\cite[Theorem~7.3]{LW}}]\label{prop:effective-iff-full}
Given an effective vertical connection $K \colon T^2M \to TM$ on $M$,
there exists a unique horizontal connection
$H \colon TM \times_M TM \to T^2M$ compatible with $K$
(i.e., such that $(K, H)$ is a full connection).
Consequently, specifying a full connection is equivalent to choosing an effective vertical connection.
\end{prop}

\begin{rem}\label{rem:H-inherits-scalar}
The construction in \cite[Theorem~7.3]{LW} is carried out for differential bundles,
whose structure involves no scalar multiplication, so the morphism $H$ produced by
Proposition~\ref{prop:effective-iff-full} need not a priori respect the scalar
multiplications. However, the horizontal connection $H$ is precisely
\[
  H = \Theta^{-1} \circ s, \qquad
  s(v_1, v_2) = \bigl(v_1,\; v_2,\; 0_M \circ \pi_M(v_1)\bigr),
\]
that is, the inverse image under $\Theta$ of the inclusion of $TM \times_M TM$ as
the first two factors. Hence, if $K$ additionally satisfies axioms (e) and (f),
then by Remark~\ref{rem:theta-bi-scalar} scaling the $\pi_{TM}$-input of $H$
corresponds to $T_{(2)}\kappa_M$ and scaling the $T(\pi_M)$-input corresponds to
$\kappa_{TM}$; consequently $H$ satisfies the scalar axioms (e) and (f) of
Definition~\ref{def:horizontal-connection}, with no further hypotheses.%
\footnote{We adopt the convention $\pi_{TM} \circ H = \pr_2$ and
$T(\pi_M) \circ H = \pr_1$ in Definition~\ref{def:horizontal-connection}
(consistent with the canonical horizontal connection on $\R^n$,
Example~\ref{ex:euclidean-horizontal-connection}); with this convention
$\kappa^{(2)}$ is paired with $T_{(2)}\kappa_M$ in axiom (f) and $\kappa^{(1)}$
with $\kappa_{TM}$ in axiom (e).}
Thus the equivalence of Proposition~\ref{prop:effective-iff-full} holds at the
level of structures preserving scalar multiplication: an effective vertical
connection satisfying (e), (f) corresponds to a full connection whose horizontal
part satisfies (e), (f).
\end{rem}

\subsection{Covariant derivatives along maps}\label{subsec:cov-along}

\begin{defn}\label{def:vf-along-map}
Let $f \colon M \to N$ be a smooth map of elastic diffeological spaces.
A smooth map $V \colon M \to TN$ is called a
\textbf{vector field along $f$} if $\pi_N \circ V = f$.
\end{defn}

\begin{defn}\label{def:covariant-derivative-along}
Let $K$ be a vertical connection on $N$, and let $V$ be a vector field
along $f \colon M \to N$.
The \textbf{covariant derivative} of $V$ in the direction of a vector field
$X \colon M \to TM$ on $M$ is defined by
\begin{align*}
    \nabla^K_X V := K \circ T(V) \circ X.
\end{align*}
When $\nabla^K_X V = 0_N \circ f$, we say that $V$ is \textbf{parallel}
in the direction of $X$.
\end{defn}
\section{Riemannian metrics and Levi-Civita connections}\label{sec:riemannian-metric}

Throughout this section, $M$ denotes an elastic diffeological space unless
stated otherwise. Tangent vectors $v, w \in TM$ are said to lie \emph{in a
common fibre} if $\pi_M(v) = \pi_M(w)$; each fibre of $\pi_M$ is an
$\R$-module with respect to the operations $+_M$ and $\kappa_M$ of
Theorem~\ref{thm:tangent-structure-2}.

\subsection{Riemannian metrics}\label{subsec:metric}

\begin{defn}\label{def:metric}
A smooth map $g \colon TM \times_M TM \to \R$ is called a
\textbf{Riemannian metric} on $M$ if it satisfies the following conditions
for all $v, w, z \in TM$ lying in a common fibre and all
$r_1, r_2 \in \R$:
 \begin{description}
     \item[Bilinearity]
   $g(r_1v + r_2w,\, z) = r_1g(v, z) + r_2g(w, z)$,\quad
   $g(v,\, r_1w + r_2z) = r_1g(v, w) + r_2g(v, z)$.
 \item[Symmetry]
   $g(v, w) = g(w, v)$ for all $(v, w) \in TM \times_M TM$.
 \item[Positivity]
   $g(v, v) \geq 0$ for all $v \in TM$.
 \item[Definiteness]
   $g(v, v) = 0$ if and only if $v = 0_M(\pi_M(v))$.
 \end{description}
\end{defn}

In \cite{KSS}, weak Riemannian metrics were defined for arbitrary
diffeological spaces as follows.

\begin{defn}[{\cite[Definition~3.1]{KSS}}]\label{def:kss-weak-metric}
  Let $M$ be a diffeological space, not necessarily elastic. A map
  $g \colon \mathbb{L}\hat{T}_2(M) \to \R$ is a
  \textbf{weak Riemannian metric} in the sense of \cite{KSS} if,
  for every $P \in \D(M)$, the map
  \begin{align*}
    g(P) := g \circ j_P \colon TU_P \times_{U_P} TU_P \to \R
  \end{align*}
  is a positive symmetric covariant $2$-tensor field on $U_P$.
  Here $j_P \colon \hat{T}_2U_P = TU_P \times_{U_P} TU_P \to \mathbb{L}\hat{T}_2(M)$
  denotes the canonical map into the colimit
  $\mathbb{L}\hat{T}_2(M) = \colim_{P \in \D(M)} \hat{T}_2U_P$,
  written $\pi_{\hat{T}_2(U_P)}$ in \cite{KSS}.
\end{defn}

\begin{defn}[{\cite[Definition~3.7]{KSS}}]\label{def:kss-definite}
  Let $M$ be a diffeological space and let $g$ be a weak Riemannian metric
  on $M$. If there exists a generating family $\mathcal{G}$ of $\D(M)$ such
  that $g(P)$ is definite in the classical sense for every
  $P \in \mathcal{G}$, then $g$ is said to be \textbf{definite} in the
  sense of \cite{KSS}.
\end{defn}

\begin{rem}\label{rem:kss-equivalence}
  Since $M$ is elastic, the isomorphism
  $\theta_{2,M} \colon \mathbb{L}\hat{T}_2(M) \xrightarrow{\ \simeq\ } TM \times_M TM$
  of axiom~(E1) identifies the two sides; under this identification the
  canonical map $j_P$ corresponds to
  $T(P) \times_{U_P} T(P)$ (Remark~\ref{rem:cocone-TP}).
  Consequently, on an elastic diffeological space the axioms of a weak
  metric in the sense of \cite{KSS} coincide with bilinearity, symmetry,
  and positivity in the sense of Definition~\ref{def:metric}.
\end{rem}

\begin{notation}\label{not:g-along-plot}
  For a plot $P \in \D(M)$ we write
  \[
    g(P) := g \circ \bigl(T(P) \times_P T(P)\bigr),
  \]
  a positive symmetric covariant $2$-tensor field on $U_P$, and
  $g(P)(u)(v, w) := g([P, u, v], [P, u, w])$ for its value at $u \in U_P$.
  This notation recurs in the energy and length functionals of
  Section~\ref{sec:geodesics}.
\end{notation}

\begin{prop}\label{prop:kss-definite-implies-definite}
  Let $g$ be a weak Riemannian metric on $M$ which is definite in the
  sense of \cite{KSS}. Then $g$, regarded as a map
  $TM \times_M TM \to \R$ via Remark~\ref{rem:kss-equivalence},
  satisfies the definiteness condition of Definition~\ref{def:metric}.
\end{prop}

\begin{proof}
  Let $[P, u, v] \in TM$ with $g([P, u, v], [P, u, v]) = 0$.
  Since the generating family $\mathcal{G}$ generates $\D(M)$, there exists
  an open neighbourhood $W \subseteq U_P$ of $u$ such that either $P|_W$ is
  constant, or $P|_W = Q \circ f$ for some generator
  $Q \in \mathcal{G}$ and some smooth map $f \colon W \to U_Q$.
  If $P|_W$ is constant, then $[P, u, v] = 0_M(P(u))$ and there is nothing
  to prove. In the latter case, functoriality
  $T(P|_W) = T(Q) \circ \hat{T}(f)$ (Remark~\ref{rem:cocone-TP}) gives
  \begin{align*}
    [P, u, v] = [P|_W, u, v] = [Q, f(u), df_u(v)].
  \end{align*}
  Hence, in terms of the pullback tensors of
  Definition~\ref{def:kss-weak-metric},
  \begin{align*}
    0 = g([P, u, v], [P, u, v])
      = g(Q)(f(u))\bigl(df_u(v),\, df_u(v)\bigr).
  \end{align*}
  Since $Q \in \mathcal{G}$, the tensor $g(Q)$ is definite, so
  $df_u(v) = 0$. Therefore
  $[P, u, v] = [Q, f(u), 0] = 0_M(P(u))$,
  which is the definiteness condition of Definition~\ref{def:metric}.
\end{proof}

Recall from Lemma~\ref{lem:cotangent-omega} the identification
$\Clin(TM, \R) \cong \Omega^1(M)$. A Riemannian metric induces a
map from vector fields to $1$-forms as follows: for $V \in \mathfrak{X}(M)$,
the composite
\begin{align*}
  g \circ \langle V \circ \pi_M,\, \id_{TM} \rangle \colon TM \to \R,
  \qquad u \longmapsto g\bigl(V(\pi_M(u)),\, u\bigr),
\end{align*}
is smooth and fibrewise linear by bilinearity of $g$, hence an element of
$\Clin(TM,\R) \cong \Omega^1(M)$.

As in classical Riemannian geometry, a metric induces a correspondence between vector fields and differential forms.

\begin{defn}\label{def:flat}
The \textbf{flat morphism} associated with $g$ is the map
\begin{align*}
  \flat \colon \mathfrak{X}(M) \to \Omega^1(M),
  \qquad
  \flat(V) := g \circ \langle V \circ \pi_M,\, \id_{TM} \rangle.
\end{align*}
\end{defn}

\begin{defn}\label{def:non-degenerate}
Let $g \colon TM \times_M TM \to \R$ be a positive, symmetric, and bilinear map.
\begin{enumerate}
    \item $g$ is \textbf{weakly non-degenerate} if for every
      $v \in TM$, the condition $g(v, u) = 0$ for all $u \in TM$ with
      $\pi_M(u) = \pi_M(v)$ implies $v = 0_M(\pi_M(v))$.
    \item $g$ is \textbf{strongly non-degenerate} if the flat morphism
      $\flat \colon \mathfrak{X}(M) \to \Omega^1(M)$ is a diffeomorphism
      with respect to the functional diffeologies. In this case, its
      inverse is written $\sharp \colon \Omega^1(M) \to \mathfrak{X}(M)$.
      (The notation $\flat$, $\sharp$ follows the classical convention.)
\end{enumerate}
\end{defn}

\begin{rem}\label{rem:weak-non-degenerate-linear}
  Weak non-degeneracy states that $g$ separates the points of each fibre:
  by bilinearity, it is equivalent to the condition that for
  $v, w \in TM$ lying in a common fibre, if $g(v, u) = g(w, u)$ for all
  $u \in TM$ with $\pi_M(u) = \pi_M(v)$, then $v = w$ (apply the
  definition to $v - w$). In particular, if $g$ is weakly non-degenerate,
  then $\flat$ is injective.
\end{rem}

\begin{prop}\label{prop:definite-implies-weak-nondegenerate}
  If $g \colon TM \times_M TM \to \R$ satisfies the definiteness condition,
  then $g$ is weakly non-degenerate.
\end{prop}

\begin{proof}
  Let $v \in TM$ satisfy $g(v, u) = 0$ for all $u \in TM$ with
  $\pi_M(u) = \pi_M(v)$.
  Taking $u = v$ gives $g(v, v) = 0$, and definiteness then implies
  $v = 0_M(\pi_M(v))$.
\end{proof}

\subsection{Levi-Civita connections}\label{subsec:levi-civita}

Throughout this subsection, $(M, g)$ denotes an \textbf{elastic Riemannian
diffeological space}, that is, an elastic diffeological space $M$ equipped
with a Riemannian metric $g$ in the sense of Definition~\ref{def:metric}.

Before defining the Levi-Civita connection, we fix notation for the
morphisms appearing in the metric compatibility condition. Since the
pullbacks $T_k$ are preserved pointwise by the tangent functor
(Section~\ref{sec:RosickysAxioms}), we identify
\begin{align*}
  T(TM \times_M TM) \;\cong\; T^2M \times_{T(\pi_M)} T^2M,
\end{align*}
and on this space we define
\begin{align*}
  K \times_{\pi} \pi_{TM} \colon (A, B) \longmapsto (K(A),\, \pi_{TM}(B)),
  \qquad
  \pi_{TM} \times_{\pi} K \colon (A, B) \longmapsto (\pi_{TM}(A),\, K(B)),
\end{align*}
both with values in $TM \times_M TM$; the base points match by axiom~(b)
of Definition~\ref{def:vertical-connection} and naturality of $\pi$.
Finally, for the smooth function $g$ on the diffeological space
$TM \times_M TM$ we write
\begin{align*}
  d(g) := \pr_2 \circ T(g) \colon T(TM \times_M TM) \to \R,
\end{align*}
which restricts along every plot $R$ of $TM \times_M TM$ to the classical
exterior derivative $d(g \circ R)$, by the computation in the proof of
Proposition~\ref{prop:pr2-Tf}.

\begin{defn}\label{def:levi-civita}
A vertical connection $K \colon T^2M \to TM$ on $M$ is called a
\textbf{Levi-Civita connection} for $g$ if it satisfies the following two
conditions:
\begin{enumerate}
    \item $K$ is torsion-free, i.e., $K \circ \tau_M = K$.
    \item \textbf{Metric compatibility}:
    \begin{equation}\label{eq:metric-compatibility}
        d(g) = +_{\R} \circ \langle g \circ (K \times_{\pi} \pi_{TM}),\;
               g \circ (\pi_{TM} \times_{\pi} K) \rangle.
    \end{equation}
\end{enumerate}
\end{defn}

\begin{prop}\label{prop:metric-compatibility-vf}
  Let $K$ be a Levi-Civita connection for $g$.
  For any vector fields $X, Y, Z \in \mathfrak{X}(M)$,
  \begin{align*}
    Xg(Y, Z) = g(\nabla^K_X Y, Z) + g(Y, \nabla^K_X Z).
  \end{align*}
\end{prop}

\begin{proof}
  We compose both sides of \eqref{eq:metric-compatibility} with
  $(T(Y)\circ X,\; T(Z) \circ X) = T(\langle Y, Z\rangle) \circ X$,
  where the equality holds under the identification
  $T(TM \times_M TM) \cong T^2M \times_{T(\pi_M)} T^2M$ above.
  For the left-hand side, functoriality of $T$,
  Proposition~\ref{prop:pr2-Tf} applied to the smooth function
  $g(Y, Z) := g \circ \langle Y, Z \rangle \colon M \to \R$, and
  Definition~\ref{def:vf-action} give
  \begin{align*}
    d(g)\circ T(\langle Y, Z\rangle) \circ X
      = \pr_2 \circ T\bigl(g(Y, Z)\bigr) \circ X
      = d\bigl(g(Y, Z)\bigr) \circ X
      = Xg(Y, Z).
  \end{align*}
  For the right-hand side,
  \begin{align*}
    +_{\R} \circ \langle g \circ (K \times_{\pi} \pi_{TM}),\;
    g \circ (\pi_{TM} \times_{\pi} K) \rangle \circ (T(Y)\circ X, T(Z) \circ X)
      &= g(K \circ T(Y)\circ X,\; \pi_{TM} \circ T(Z) \circ X) \\
      &\quad + g(\pi_{TM} \circ T(Y) \circ X,\; K \circ T(Z)\circ X).
  \end{align*}
  Since $\pi$ is a natural transformation,
  $\pi_{TM} \circ T(Z) \circ X = Z \circ \pi_M \circ X = Z$,
  and similarly $\pi_{TM} \circ T(Y) \circ X = Y$.
  Therefore the right-hand side equals
  $g(\nabla^K_X Y, Z) + g(Y, \nabla^K_X Z)$, which proves the claim.
\end{proof}

\begin{prop}\label{prop:vert-hori-compatibility}
Let $(K, H)$ be a full connection on $M$ whose vertical part $K$ satisfies
the metric compatibility condition \eqref{eq:metric-compatibility}. Then
\begin{align*}
    d(g) \circ (H \times H) = 0,
\end{align*}
where $H \times H$ is defined on the fibre product
$(TM \times_M TM) \times_{\pr_1} (TM \times_M TM)$.
\end{prop}

\begin{proof}
We compose both sides of \eqref{eq:metric-compatibility} on the right with $H \times H$; note that axiom~(a) of horizontal connections, $T(\pi_M) \circ H = \pr_1$, guarantees that $H \times H$ takes values in $T^2M \times_{T(\pi_M)} T^2M \cong T(TM \times_M TM)$.
By axiom~(a) of full connections,
\[
K \circ H = 0_M \circ \pi_M \circ \pr_1.
\]
Using further the property $\pi_{TM} \circ H = \pr_2$ of horizontal connections, each component of the right-hand side satisfies
\[
g \circ (K \times_\pi \pi_{TM}) \circ (H \times H) = 0,
\qquad
g \circ (\pi_{TM} \times_\pi K) \circ (H \times H) = 0,
\]
since $g(0_M(x), u) = 0$ for any $u$ in the fibre over $x$, by
bilinearity. Hence
\[
d(g) \circ (H \times H)
= +_{\R} \circ \langle 0,\; 0 \rangle
= 0. \qedhere
\]
\end{proof}

Using the vector field formulation of metric compatibility, a computation
analogous to that in ordinary Riemannian geometry yields the following
Koszul formula.

\begin{prop}[Koszul formula]\label{prop:koszul-formula}
  Let $K$ be a Levi-Civita connection for $g$.
  For any vector fields $X, Y, Z \in \mathfrak{X}(M)$,
  \begin{equation}\label{eq:koszul}
  \begin{aligned}
    2g(\nabla^K_X Y,\; Z)
      &= X\bigl(g(Y,Z)\bigr) + Y\bigl(g(X,Z)\bigr) - Z\bigl(g(X,Y)\bigr) \\
      &\quad + g([X,Y],Z) - g([X,Z],Y) - g([Y,Z],X).
  \end{aligned}
  \end{equation}
\end{prop}

\begin{proof}
  Applying Proposition~\ref{prop:metric-compatibility-vf} three times,
  \begin{align*}
    X\bigl(g(Y,Z)\bigr) &= g(\nabla^K_X Y, Z) + g(Y, \nabla^K_X Z), \\
    Y\bigl(g(X,Z)\bigr) &= g(\nabla^K_Y X, Z) + g(X, \nabla^K_Y Z), \\
    Z\bigl(g(X,Y)\bigr) &= g(\nabla^K_Z X, Y) + g(X, \nabla^K_Z Y).
  \end{align*}
  Adding the first two identities, subtracting the third, and collecting
  terms by bilinearity of $g$, we obtain
  \begin{align*}
    &X\bigl(g(Y,Z)\bigr) + Y\bigl(g(X,Z)\bigr) - Z\bigl(g(X,Y)\bigr) \\
      &\qquad = g(\nabla^K_X Y + \nabla^K_Y X,\; Z)
        + g(\nabla^K_X Z - \nabla^K_Z X,\; Y)
        + g(\nabla^K_Y Z - \nabla^K_Z Y,\; X).
  \end{align*}
  Since $K$ is torsion-free, Proposition~\ref{prop:torsion-free-bracket}
  gives $\nabla^K_X Z - \nabla^K_Z X = [X, Z]$,
  $\nabla^K_Y Z - \nabla^K_Z Y = [Y, Z]$, and
  $\nabla^K_X Y + \nabla^K_Y X = 2\nabla^K_X Y - [X, Y]$. Substituting,
  \begin{align*}
    X\bigl(g(Y,Z)\bigr) + Y\bigl(g(X,Z)\bigr) - Z\bigl(g(X,Y)\bigr)
      = 2g(\nabla^K_X Y, Z) - g([X,Y], Z) + g([X,Z], Y) + g([Y,Z], X),
  \end{align*}
  which rearranges to \eqref{eq:koszul}.
\end{proof}

\begin{thm}\label{thm:uniqueness-cov-deriv}
  If $K_1$ and $K_2$ are Levi-Civita connections for $g$, then
  $\nabla^{K_1} = \nabla^{K_2}$.
\end{thm}

\begin{proof}
  Fix $X, Y \in \mathfrak{X}(M)$ and let
  $D := \nabla^{K_1}_X Y - \nabla^{K_2}_X Y \in \mathfrak{X}(M)$,
  the difference being taken in the fibrewise abelian group structure of
  $TM$. The right-hand side of the Koszul formula \eqref{eq:koszul} does
  not involve $K$, so
  $g(\nabla^{K_1}_X Y, Z) = g(\nabla^{K_2}_X Y, Z)$, and hence by
  bilinearity $g(D, Z) = 0$, for every $Z \in \mathfrak{X}(M)$.
  Since $D$ is itself a vector field, we may substitute $Z = D$.
  This yields $g(D(x), D(x)) = 0$ for every $x \in M$, and
  definiteness gives $D(x) = 0_M(x)$. Therefore
  $\nabla^{K_1}_X Y = \nabla^{K_2}_X Y$.
\end{proof}

\begin{rem}\label{rem:uniqueness-weak-metric}
  The proof above uses only that the vector field $D$ itself may be
  inserted as $Z$; in particular it does not require every tangent vector
  to be realized by a global vector field. Definiteness may moreover be
  relaxed: if $g$ is merely positive, symmetric, and bilinear (a weak
  metric) and weakly non-degenerate, the same conclusion holds. Indeed,
  positivity implies the fibrewise Cauchy--Schwarz inequality
  $g(v, u)^2 \leq g(v, v)\, g(u, u)$ for $u, v \in TM$ lying in a common
  fibre, so $g(D(x), D(x)) = 0$ forces $g(D(x), u) = 0$ for all
  $u \in TM$ with $\pi_M(u) = x$, and weak non-degeneracy gives
  $D(x) = 0_M(x)$.
\end{rem}

\begin{rem}\label{rem:lc-existence}
  Theorem~\ref{thm:uniqueness-cov-deriv} concerns the covariant derivative.
  We do not know whether a Levi-Civita connection exists for a general
  Riemannian metric on an elastic diffeological space, even when the metric
  is strongly non-degenerate, nor whether it is unique as a morphism
  $T^2M \to TM$ (cf.\ Remark~\ref{rem:cov-deriv-vs-connection}), nor whether
  the converse of Proposition~\ref{prop:vert-hori-compatibility} holds. The
  examples of Section~\ref{sec:examples} are spaces on which a Levi-Civita
  connection is available.
\end{rem}

\begin{ex}[Riemannian manifolds]\label{ex:riemannian-manifold}
  Let $X$ be a finite-dimensional manifold with a linear connection whose
  Christoffel symbols in a chart are $\Gamma$. In the coordinates of
  Example~\ref{ex:euclidean-vertical-connection}, its connector
  \begin{align*}
    K(u, v_0, v_1, v_{01}) = \bigl(u,\; v_{01} + \Gamma(u)(v_0, v_1)\bigr)
  \end{align*}
  is a well-defined morphism $T^2X \to TX$. Since $T$ is local
  (Proposition~\ref{prop:tangent-structure-1}), the axioms of
  Definition~\ref{def:vertical-connection} may be checked chart by chart,
  where they hold by Example~\ref{ex:euclidean-vertical-connection}; thus $K$
  is a vertical connection. It is effective
  (Definition~\ref{def:effective-vertical-connection}), since in each chart
  $\Theta(u, v_0, v_1, v_{01}) = \bigl((u, v_0), (u, v_1), (u, v_{01} +
  \Gamma(u)(v_0, v_1))\bigr)$ is invertible; the horizontal connection
  forming a full connection with it (Proposition~\ref{prop:effective-iff-full})
  is
  \begin{align*}
    H\bigl((u, v_0), (u, v_1)\bigr)
      = \bigl(u,\; v_1,\; v_0,\; -\Gamma(u)(v_1, v_0)\bigr),
  \end{align*}
  as one checks from $\pi_{TX} \circ H = \pr_2$, $T(\pi_X) \circ H = \pr_1$ and
  $K \circ H = 0$. The geodesic spray $S(u, v) = (u, v, v, -\Gamma(u)(v, v))$
  is the classical one, so $S$ is complete if and only if every geodesic is
  defined on all of $\R$. If the connection is the Levi-Civita connection of
  a Riemannian metric $g$, then $K$ is torsion-free because $\Gamma$ is
  symmetric, and a computation in local coordinates shows that
  \eqref{eq:metric-compatibility} is equivalent to the classical identity
  $\nabla g = 0$; hence $K$ is a Levi-Civita connection for $g$ in the sense of
  Definition~\ref{def:levi-civita}.
\end{ex}

\section{Curvature}\label{sec:curvature}

In this section, following \cite[Section~3.16]{CC17}, we introduce the curvature
morphism and curvature tensor on elastic diffeological spaces.

\begin{defn}[{\cite[Definition~3.20]{CC17}}]\label{def:curvature-morphism}
  Let $M$ be an elastic diffeological space and $K \colon T^2M \to TM$ a vertical
  connection on $M$. The \textbf{curvature morphism} $C_K \colon T^3M \to TM$ of $K$
  is defined by
  \begin{align*}
    C_K := K \circ T(K) \circ \tau_{TM} - K \circ T(K).
  \end{align*}
 Both composites $K \circ T(K) \circ \tau_{TM}$ and $K \circ T(K)$ are morphisms $T^3M \to TM$ lying over the same morphism $\pi_M \circ \pi_{TM} \circ \pi_{T^2M} \colon T^3M \to M$, by axiom~(b) of Definition~\ref{def:vertical-connection} and naturality of $\pi$; their difference is therefore formed in the fibrewise abelian group structure of the bundle $\pi_M \colon TM \to M$, whose fibres admit negatives, being $\R$-modules. 
 The equation $C_K = 0$ means that $C_K$ factors through the zero section, i.e.\ $C_K = 0_M \circ \pi_M \circ \pi_{TM} \circ \pi_{T^2M}$; in this case the connection $K$ is said to be \textbf{flat}.
\end{defn}

\begin{ex}\label{ex:flat-euclidean}
  Let $M = \R^n$ with the canonical affine vertical connection
  $K_0(u, v_0, v_1, v_{01}) = (u, v_{01})$ of
  Example~\ref{ex:euclidean-vertical-connection}. Extending the coordinates
  of that example, identify $T^3\R^n \cong (\R^n)^{8}$ with coordinates
  $(v_S)_{S \subseteq \{0,1,2\}}$, where $v_\emptyset = u$ and the
  subscripts record the tangent directions; the first four components
  $(u, v_0, v_1, v_{01})$ form the base point in $T^2\R^n$ and the last
  four $(v_2, v_{02}, v_{12}, v_{012})$ its derivative in the new
  direction $2$. The morphism $\tau_{TM}$, being the component of the
  canonical flip at the object $TM$, exchanges the two outer directions
  $1$ and $2$ while fixing the direction $0$ internal to $TM$:
  \begin{align*}
    \tau_{TM}(v_S)_S = (v_{\sigma(S)})_S, \qquad \sigma = (1\ 2).
  \end{align*}
  Since $K_0$ is linear, its tangent map acts by $K_0$ on both blocks:
  \begin{align*}
    T(K_0)(v_S)_S = (u,\, v_{01},\, v_2,\, v_{012}) \in T^2\R^n,
  \end{align*}
  whence $K_0 \circ T(K_0) = (u, v_{012})$. Precomposing with
  $\tau_{TM}$ replaces $(v_1, v_{01})$ by $(v_2, v_{02})$ and vice versa,
  so
  \begin{align*}
    K_0 \circ T(K_0) \circ \tau_{TM}
      = K_0(u,\, v_{02},\, v_1,\, v_{012}) = (u, v_{012}).
  \end{align*}
  The two composites coincide, and therefore $C_{K_0} = 0$: the canonical
  connection on $\R^n$ is flat.
\end{ex}

\begin{defn}[{\cite[Definition~3.21]{CC17}}]\label{def:curvature-tensor}
  Let $X, Y, Z \in \mathfrak{X}(M)$ be vector fields on $M$.
  The \textbf{curvature tensor} $R^K(X,Y,Z) \in \mathfrak{X}(M)$
  is defined using the curvature morphism $C_K$ by
  \begin{align*}
    R^K(X,Y,Z) := C_K \circ T^2(Z) \circ T(X) \circ Y.
  \end{align*}
\end{defn}

This curvature tensor admits an expression in terms of covariant derivatives,
as in ordinary Riemannian geometry.

\begin{prop}[{\cite[Proposition~3.22]{CC17}}]\label{prop:curvature-covariant-derivative}
  For vector fields $X, Y, Z \in \mathfrak{X}(M)$ on $M$, the following holds:
  \begin{align*}
    R^K(X,Y,Z)
      = \nabla^K_X (\nabla^K_Y Z)
        - \nabla^K_Y (\nabla^K_X Z)
        - \nabla^K_{[X,Y]} Z.
  \end{align*}
\end{prop}

\begin{rem}\label{rem:curvature-pending}
  The curvature recalled above enjoys the familiar properties of the classical
  theory, established by Cockett and Cruttwell for an arbitrary tangent
  category: the expression of $R^K$ in terms of the covariant derivative
  (Proposition~\ref{prop:curvature-covariant-derivative}), the identities for
  curvature and torsion, and, for torsion-free connections, the Bianchi
  identities \cite[Theorem~3.34]{CC17}. These results require the tangent
  bundles to admit negatives, which holds in $\mathsf{Elast}$ since the fibres
  of $TM$ are $\R$-modules; they therefore apply verbatim, and we do not
  reproduce the proofs.
\end{rem}

\section{Geodesics and the energy functional}\label{sec:geodesics}

This section contains the main results of the paper.
Section~\ref{subsec:parallel-transport} recalls from \cite{CCL} the notions of
dynamical system and curve object, and
Sections~\ref{subsec:geodesic-def}--\ref{subsec:geodesic-flow} the geodesic
spray and geodesic flow of a full connection \cite[Section~3.5]{CCL}. We then
introduce the class $\mathsf{Elast}_{\mathrm{curve}}$ and prove existence and
uniqueness of geodesics (Section~\ref{subsec:geodesic-existence}), establish
the first variation formula (Section~\ref{subsec:energy}), prove that critical
paths of the energy are geodesics (Section~\ref{subsec:critical}), and deduce
that length-minimizing paths are geodesics (Section~\ref{subsec:length});
together these prove Theorems~\ref{mainthm:A} and~\ref{mainthm:B}
(Section~\ref{subsec:proof-main}).

Throughout this section, $M$ denotes an elastic diffeological space, and
$(K, H)$ a full connection on $M$ (Definition~\ref{def:full-connection}),
unless further hypotheses are stated. We write $\Path(M)$ for the
diffeological space $C^\infty(\R, M)$ of smooth paths, with the functional
diffeology, and for $x, y \in M$ we write $\Path(M; x, y)$ for the
subspace of paths $\gamma$ with $\gamma(0) = x$ and $\gamma(1) = y$.

\subsection{Dynamical systems and curve objects}\label{subsec:parallel-transport}

We recall the notions of dynamical system, solution, and curve object from
\cite{CCL}, specialized to $\mathsf{Elast}$.

\begin{defn}[{\cite[Definition~2.24]{CCL}}]\label{def:dynamical-system}
  Let $V$ be a vector field on $M$ and $x \in M$.
  The triple $(M, V, x)$ is called a \textbf{dynamical system};
  $V$ is the \textbf{differential transition} and $x$ is the \textbf{initial state}.
  A \textbf{morphism} $f \colon (M, V, x) \to (M', V', x')$ of dynamical systems
  is a smooth map $f \colon M \to M'$ satisfying $f(x) = x'$ and
  $T(f) \circ V = V' \circ f$; that is, the following diagram commutes:
  \[ \xymatrix{1 \ar[r]^x \ar[dr]_{x'} & M \ar[d]_f \ar[r]^V & T(M) \ar[d]^{T(f)} \\
               & M' \ar[r]_{V'} & T(M') } \]
\end{defn}

\begin{rem}\label{rem:initial-state}
  In a general tangent category, the initial state is a morphism
  $x \colon \mathbf{1} \to M$ from the terminal object $\mathbf{1}$
  (\cite[Definition~2.24]{CCL}).
  In elastic diffeology, since each object is a set, one may take points directly;
  however, individual points can also be viewed as morphisms from $\mathbf{1}$
  when convenient.
\end{rem}

\begin{defn}[{\cite[Definition~3.1]{CCL}}]\label{def:parametrized-dynamical-system}
  Let $V$ be a vector field on $M$ and $g \colon X \to M$ a smooth map.
  The triple $(M, V, g)$ is called a \textbf{parametrized dynamical system},
  and $g$ is the \textbf{parametrized initial state}.
\end{defn}

\begin{defn}[{\cite[Definition~3.1]{CCL}}]\label{def:parametrized-solution}
  Let $C$ be an elastic diffeological space, $c_1 \colon C \to TC$ a vector field
  on $C$, and $c_0 \in C$. A \textbf{solution} to the parametrized dynamical system
  $(M, V, g)$ is a smooth map $\gamma \colon C \times X \to M$ satisfying
  the initial condition $\gamma(c_0, p) = g(p)$ and the differential condition
  $T(\gamma) \circ (c_1 \times 0_X) = V \circ \gamma$.
  That is, the following diagram commutes:
 \[
   \xymatrix{
     X \ar[r]^{\langle c_0, 1\rangle} \ar[dr]_{g}
       & C \times X \ar[r]^{c_1 \times 0} \ar[d]_{\gamma}
       & T(C \times X) \ar[d]^{T(\gamma)} \\
     & M \ar[r]_{V}
       & TM
   }
 \]
\end{defn}

\begin{rem}\label{rem:point-as-constant-map}
  As in Remark~\ref{rem:initial-state}, $c_0 \in C$ may also be regarded, for
  any object $X$, as the constant map $X \to C$ sending every point to $c_0$;
  we write this map simply as $c_0$ as well.
\end{rem}

\begin{defn}[{\cite[Definition~4.1]{CCL}}]\label{def:curve-object}
  A dynamical system $(C, c_1, c_0)$ is a \textbf{curve object} if it satisfies
  the following:
  \begin{enumerate}
    \item \textbf{Preinitiality}: For any parametrized dynamical system $(M, V, g)$,
      a solution $\gamma \colon C \times X \to M$, if it exists, is unique:
    \[
      \xymatrix{
        X \ar[r]^{\langle c_0, 1\rangle} \ar[dr]_{g}
          & C \times X \ar[r]^{c_1 \times 0} \ar@{-->}[d]_{\gamma}
          & T(C \times X) \ar@{-->}[d]^{T(\gamma)} \\
        & M \ar[r]_{V}
          & TM
      }
    \]
    \item \textbf{Self-commutativity}: $\tau_C \circ T(c_1) \circ c_1
      = T(c_1) \circ c_1$.
    \item \textbf{Completeness}: The parametrized dynamical system
      $(C, c_1, \id_C)$ on $C$ has a unique solution $\sigma \colon C \times C \to C$:
    \[
      \xymatrix{
        C \ar[r]^{\langle c_0, 1\rangle} \ar[dr]_{\id_C}
          & C \times C \ar[r]^{c_1 \times 0} \ar@{-->}[d]_{\sigma}
          & T(C \times C) \ar@{-->}[d]^{T(\sigma)} \\
        & C \ar[r]_{c_1}
          & TC
      }
    \]
  \end{enumerate}
\end{defn}

\begin{defn}[{\cite[Definition~4.9]{CCL}}]\label{def:complete-vector-field}
  A vector field $V$ on $M$ is called \textbf{complete} if the parametrized
  dynamical system $(M, V, \id_M)$ with parametrized initial state $\id_M$
  has a solution $\gamma \colon C \times M \to M$:
  \[
    \xymatrix{
      M \ar[r]^{\langle c_0, 1\rangle} \ar[dr]_{\id_M}
        & C \times M \ar[r]^{c_1 \times 0} \ar@{-->}[d]_{\gamma}
        & T(C \times M) \ar@{-->}[d]^{T(\gamma)} \\
      & M \ar[r]_{V}
        & TM
    }
  \]
\end{defn}

In this paper $(C, c_1, c_0) = (\R, \partial, 0)$. Since smooth dynamical
systems on convenient vector spaces may have no solution or several
\cite[Remark~4.5]{CCL}, $(\R, \partial, 0)$ is not a curve object in
$\mathsf{Elast}$. The definitions above and below, and
Proposition~\ref{prop:higher-dynamical-system}, only use that
$(\R, \partial, 0)$ is a dynamical system; preinitiality enters only through
Definition~\ref{def:elast-lc}.

\subsection{Geodesics}\label{subsec:geodesic-def}

Let $\gamma$ be a path on $M$, and let $\partial$ denote the unit vector field
on $\R$. We call $\dot{\gamma} = T(\gamma) \circ \partial = [\gamma, t, \partial]$
the \textbf{velocity vector field} of $\gamma$.

\begin{rem}\label{rem:velocity-along}
  $\dot{\gamma}$ is a vector field along $\gamma$
  (Definition~\ref{def:vf-along-map}). Indeed, by naturality of $\pi$,
  \begin{align*}
    \pi_{M} \circ \dot{\gamma}
      = \pi_{M} \circ T(\gamma) \circ \partial
      = \gamma \circ \pi_{\R} \circ \partial
      = \gamma \circ \id_{\R}
      = \gamma.
  \end{align*}
\end{rem}

\begin{defn}\label{def:geodesic}
Let $K$ be a vertical connection on $M$.
A path $\gamma \colon \R \to M$ is called a \textbf{geodesic} for $K$ if
\begin{align*}
   \nabla^K_{\partial} \dot{\gamma} = 0_M \circ \gamma,
\end{align*}
where $\nabla^K_\partial \dot\gamma = K \circ T(\dot\gamma) \circ \partial$ is
the covariant derivative of the vector field $\dot\gamma$ along $\gamma$
(Definition~\ref{def:covariant-derivative-along}). For a subset
$I \subseteq \R$, we say that $\gamma$ \textbf{satisfies the geodesic equation
on $I$} if $(\nabla^K_\partial \dot\gamma)(t) = 0_M(\gamma(t))$ for all
$t \in I$.
\end{defn}

The energy functional of Section~\ref{subsec:energy} only involves the
restriction of a path to $[0,1]$, which is why the geodesic equation on $[0,1]$
appears in Theorem~\ref{mainthm:A}.

\subsection{The geodesic spray and the geodesic flow}\label{subsec:geodesic-flow}

\begin{defn}[{\cite[Definition~3.10]{CCL}}]\label{def:nth-order-dynamical-system}
  Let $n \geq 1$ be an integer.
  An \textbf{$n$th-order dynamical system} on $M$ is a triple
  $(M, V, g)$ consisting of an initial state $g \colon X \to T^{n-1}M$ and
  an $n$th-order vector field, i.e., a smooth map
  $V \colon T^{n-1}M \to T^nM$, satisfying
  \begin{align*}
      \pi_{T^{n-1}M} \circ V
        = T(\pi_{T^{n-2}M}) \circ V
        = \cdots
        = T^{n-1}(\pi_M) \circ V
        = \id_{T^{n-1}M}.
  \end{align*}
  When $n = 1$ this coincides with the dynamical system of
  Definition~\ref{def:dynamical-system}.
  Moreover, if $(M, V, g)$ is an $n$th-order dynamical system,
  then $(TM, V, g)$ is an $(n-1)$th-order dynamical system.
\end{defn}

\begin{defn}[{\cite[Definition~3.11]{CCL}}]\label{def:nth-derivative}
  Let $(C, c_1, c_0)$ be a dynamical system; in this paper always
  $(C, c_1, c_0) = (\R, \partial, 0)$.
  For any smooth map $F \colon C \times X \to M$,
  the \textbf{$n$th derivative} $F^{(n)} \colon C \times X \to T^nM$ is
  defined inductively by $F^{(0)} := F$ and, for each $n \geq 1$,
  \begin{align*}
      F^{(n)} := T(F^{(n-1)}) \circ (c_1 \times 0_X).
  \end{align*}
\end{defn}

\begin{rem}\label{rem:first-derivative}
  $F^{(1)}$ is nothing other than $T(F) \circ (c_1 \times 0_X)$,
  the left-hand side of the differential condition
  in the definition of a solution to a first-order dynamical system.
\end{rem}

\begin{defn}[{\cite[Definition~3.12]{CCL}}]\label{def:nth-solution}
  A \textbf{solution} to the $n$th-order dynamical system $(M, V, g)$
  is a smooth map $F \colon C \times X \to M$ satisfying the
  \textbf{initial condition}
  $F^{(n-1)} \circ \langle c_0, \id_X \rangle = g$ and the
  \textbf{differential condition} $F^{(n)} = V \circ F^{(n-1)}$.
\end{defn}

\begin{prop}[{\cite[Proposition~3.13]{CCL}}]\label{prop:higher-dynamical-system}
  Let $(M, V, g)$ be an $n$th-order dynamical system.
  If $\tilde{F} \colon C \times X \to TM$ is a solution to the
  $(n-1)$th-order dynamical system $(TM, V, g)$,
  then $F := \pi_M \circ \tilde{F}$ is a solution to the
  $n$th-order dynamical system $(M, V, g)$. In particular
  $\pi_M \circ F^{(1)} = F$.
\end{prop}

\begin{rem}[{\cite[Corollary~3.14]{CCL}}]\label{rem:reduction-to-first-order}
  By applying the above proposition repeatedly, the problem of solving any
  $n$th-order dynamical system $(M, V, g)$ reduces to that of solving
  the first-order dynamical system $(T^{n-1}M, V, g)$ on the iterated tangent bundle.
\end{rem}

\begin{defn}[{\cite[Section~3.5]{CCL}}]\label{def:geodesic-spray}
  Let $(K, H)$ be a full connection on $M$.
  Define the map $S \colon TM \to T^2M$ by
  \begin{align*}
      S := H \circ \langle \id_{TM}, \id_{TM} \rangle.
  \end{align*}
  Since $S$ satisfies $\pi_{TM} \circ S = T(\pi_M) \circ S = \id_{TM}$,
  it is a second-order vector field on $M$;
  we call it the \textbf{geodesic spray} of $(K, H)$.

  A solution $\Phi \colon C \times TM \to M$ to the second-order dynamical system
  $(M, S, \id_{TM})$ with initial state $\id_{TM} \colon TM \to TM$
  is called the \textbf{geodesic flow} associated with $(K, H)$.
  We say that $S$ is \textbf{complete} if it is complete as a vector field on
  $TM$ (Definition~\ref{def:complete-vector-field}); by
  Proposition~\ref{prop:higher-dynamical-system}, a geodesic flow then exists.
\end{defn}

\begin{defn}\label{def:fix-zero}
  A geodesic flow $\Phi \colon \R \times TM \to M$ \textbf{fixes the zero
  section} if $\Phi(s, 0_M(x)) = x$ for all $s \in \R$ and $x \in M$.
\end{defn}

The constant path at $x$ is always a geodesic with initial velocity $0_M(x)$
(see the proof of Proposition~\ref{prop:flow-zero}), and so is
$s \mapsto \Phi(s, 0_M(x))$. Without a uniqueness statement for geodesics these
two need not coincide, so Definition~\ref{def:fix-zero} is a genuine condition
on $\Phi$.

By definition, the geodesic flow satisfies $\Phi^{(2)} = S \circ \Phi^{(1)}$;
equivalently, the following diagram commutes:
\[
  \xymatrix{
    C \times TM \ar[d]_{\Phi^{(1)}} \ar[dr]^{\Phi^{(2)}} & \\
    TM \ar[r]_{S} & T^2M
  }
\]

\begin{prop}[{\cite[Section~3.5]{CCL}}]\label{prop:acceleration-zero}
  The geodesic flow $\Phi \colon C \times TM \to M$ satisfies
  \begin{align*}
      K \circ \Phi^{(2)} = 0_{M} \circ \Phi.
  \end{align*}
\end{prop}

\begin{proof}
  By the differential condition and Definition~\ref{def:geodesic-spray},
  \begin{align*}
      K \circ \Phi^{(2)}
        = K \circ S \circ \Phi^{(1)}
        = K \circ H \circ \langle \id_{TM}, \id_{TM} \rangle \circ \Phi^{(1)}.
  \end{align*}
  By axiom~(a) of full connections,
  $K \circ H = 0_M \circ \pi_M \circ \pr_1$, so
  \begin{align*}
      K \circ \Phi^{(2)}
        = 0_M \circ \pi_M \circ \pr_1
          \circ \langle \id_{TM}, \id_{TM} \rangle \circ \Phi^{(1)}
        = 0_M \circ \pi_M \circ \Phi^{(1)}
        = 0_M \circ \Phi,
  \end{align*}
  the last equality by $\pi_M \circ \Phi^{(1)} = \Phi$
  (Proposition~\ref{prop:higher-dynamical-system}).
\end{proof}

The geodesic equation is equivalent to the second-order equation defined by the
spray. This uses the fibre product decomposition of $T^2M$ given by a full
connection.

\begin{lem}\label{lem:geodesic-spray-equation}
  Let $(K, H)$ be a full connection on $M$ with geodesic spray $S$. A path
  $\gamma \colon \R \to M$ is a geodesic for $K$ if and only if
  $\gamma^{(2)} = S \circ \gamma^{(1)}$, where $\gamma^{(1)} = \dot\gamma$ and
  $\gamma^{(2)} = T(\dot\gamma) \circ \partial$
  (Definition~\ref{def:nth-derivative}).
\end{lem}

\begin{proof}
  By Proposition~\ref{prop:full-connection-fibre-product}, the map
  $\Theta = \langle \pi_{TM}, T(\pi_M), K \rangle \colon T^2M \to
  TM \times_M TM \times_M TM$ is an isomorphism; in particular, two maps into
  $T^2M$ coincide as soon as their composites with $\pi_{TM}$, $T(\pi_M)$ and
  $K$ coincide. For $\gamma^{(2)}$ we have
  $\pi_{TM} \circ T(\dot\gamma) \circ \partial = \dot\gamma \circ \pi_\R \circ
  \partial = \dot\gamma$ by naturality of $\pi$,
  $T(\pi_M) \circ T(\dot\gamma) \circ \partial = T(\pi_M \circ \dot\gamma)
  \circ \partial = T(\gamma) \circ \partial = \dot\gamma$ by
  Remark~\ref{rem:velocity-along}, and
  $K \circ \gamma^{(2)} = \nabla^K_\partial \dot\gamma$. For
  $S \circ \dot\gamma$, axioms~(a) and~(b) of
  Definition~\ref{def:horizontal-connection} give
  $\pi_{TM} \circ S = T(\pi_M) \circ S = \id_{TM}$, and axiom~(a) of
  Definition~\ref{def:full-connection} gives
  $K \circ S = 0_M \circ \pi_M \circ \pr_1 \circ \langle \id_{TM}, \id_{TM}
  \rangle = 0_M \circ \pi_M$; hence the three composites of
  $S \circ \dot\gamma$ are $\dot\gamma$, $\dot\gamma$ and $0_M \circ \gamma$.
  Therefore $\gamma^{(2)} = S \circ \dot\gamma$ if and only if
  $\nabla^K_\partial \dot\gamma = 0_M \circ \gamma$.
\end{proof}

\subsection{Existence and uniqueness of geodesics}\label{subsec:geodesic-existence}

The subcategory on which $\R$ serves as a curve object plays a distinguished
role in what follows; we name it.

\begin{defn}\label{def:elast-lc}
  Let $\mathsf{Elast}_{\mathrm{curve}}$ denote the full subcategory of $\mathsf{Elast}$ consisting of those elastic diffeological spaces $M$ such that, for every $n \geq 0$, the dynamical system $(\R, \partial, 0)$
  is preinitial for the iterated tangent bundle $T^nM$ (Definition~\ref{def:curve-object}); that is, every parametrized dynamical system on $T^nM$ with domain $\R$ admits at most one solution.
\end{defn}

\begin{rem}\label{rem:elast-curve-closure}\label{rem:manifolds-curve}\label{rem:elast-curve-role}
  Conditions (2) and (3) of Definition~\ref{def:curve-object} hold
  automatically for $(\R, \partial, 0)$, so $M \in
  \mathsf{Elast}_{\mathrm{curve}}$ if and only if $(\R, \partial, 0)$ is a
  curve object for every $T^nM$; imposing this on all $T^nM$ makes the class
  stable under $T$ and matches the reduction of
  Remark~\ref{rem:reduction-to-first-order}. Every finite-dimensional
  manifold belongs to $\mathsf{Elast}_{\mathrm{curve}}$: its iterated tangent
  bundles are manifolds, on which integral curves are unique
  \cite[Example~4.2]{CCL}, and two solutions $\R \times X \to T^nM$ of a
  parametrized system agree after restriction to each point of $X$. The same
  holds for Banach manifolds by \cite[Example~4.4]{CCL} and
  \cite[Sections~4.1, 4.2]{M}, and $\mathsf{Elast}_{\mathrm{curve}}$ is closed
  under quotients by principal actions of diffeologically discrete groups
  \cite[Theorem~4.8]{ShioQ}. The condition gives uniqueness, not existence: in
  Theorem~\ref{thm:geodesic-existence} the global geodesic flow comes from the
  completeness of the spray, while $\mathsf{Elast}_{\mathrm{curve}}$ yields
  the uniqueness of geodesics and the zero-section property
  (Proposition~\ref{prop:flow-zero}).
\end{rem}

As in ordinary Riemannian geometry,
the existence of geodesics is guaranteed by the geodesic flow.

\begin{thm}\label{thm:geodesic-existence}
  Let $M \in \mathsf{Elast}_{\mathrm{curve}}$ and let $(K, H)$ be a full connection
  on $M$ whose geodesic spray
  $S = H \circ \langle \id_{TM}, \id_{TM} \rangle$ is complete.
  Then for every $v_0 \in TM$ there exists a unique
  geodesic $\gamma \colon \R \to M$ for $K$ satisfying
  \begin{align*}
      \nabla^K_{\partial} \dot{\gamma} = 0_{M} \circ \gamma,
      \qquad
      \dot{\gamma}(0) = v_0.
  \end{align*}
\end{thm}

\begin{proof}
  The argument follows \cite[Section~3.5]{CCL}.
  \emph{Existence.} Since $S$ is complete, the second-order dynamical system
  $(M, S, \id_{TM})$ has a geodesic flow $\Phi \colon \R \times TM \to M$
  (Definition~\ref{def:geodesic-spray}), so that
  $\Phi^{(1)} \circ \langle 0, \id_{TM} \rangle = \id_{TM}$ and
  $\Phi^{(2)} = S \circ \Phi^{(1)}$. For the given $v_0 \in TM$, let
  $\iota_{v_0} \colon \R \to \R \times TM$ be the slice map $t \mapsto (t, v_0)$,
  and define $\gamma := \Phi \circ \iota_{v_0}$. Decomposing
  $\iota_{v_0} = \langle \id_\R, c_{v_0} \rangle$, the product-preservation of
  $T$ and the fact that $T(c_{v_0}) \circ \partial = 0_{TM} \circ c_{v_0}$ give
  \begin{align*}
    T(\iota_{v_0}) \circ \partial
      = \langle \partial,\; 0_{TM} \circ c_{v_0} \rangle
      = (\partial \times 0_{TM}) \circ \iota_{v_0}.
  \end{align*}
  Hence $\gamma^{(1)} = T(\Phi) \circ T(\iota_{v_0}) \circ \partial =
  \Phi^{(1)} \circ \iota_{v_0}$, and in the same way
  $\gamma^{(2)} = T(\Phi^{(1)}) \circ T(\iota_{v_0}) \circ \partial =
  \Phi^{(2)} \circ \iota_{v_0}$. Therefore
  $\gamma^{(2)} = S \circ \Phi^{(1)} \circ \iota_{v_0} = S \circ \gamma^{(1)}$,
  and $\gamma$ is a geodesic by Lemma~\ref{lem:geodesic-spray-equation}.
  Moreover $\dot\gamma(0) = \Phi^{(1)}(0, v_0) = v_0$ by the initial condition.

  \emph{Uniqueness.} Let $\gamma$ and $\sigma$ be geodesics with
  $\dot\gamma(0) = \dot\sigma(0) = v_0$. By
  Lemma~\ref{lem:geodesic-spray-equation},
  $T(\dot\gamma) \circ \partial = S \circ \dot\gamma$ and
  $T(\dot\sigma) \circ \partial = S \circ \dot\sigma$. Since
  $\pi_{TM} \circ S = \id_{TM}$, the map $S$ is a vector field on $TM$, and
  $\dot\gamma, \dot\sigma \colon \R \cong \R \times \ast \to TM$ are both
  solutions of the parametrized dynamical system $(TM, S, v_0)$ with parameter
  object the point $\ast$ (Definition~\ref{def:parametrized-solution}). Since
  $M \in \mathsf{Elast}_{\mathrm{curve}}$, the dynamical system
  $(\R, \partial, 0)$ is preinitial for $TM$, so $\dot\gamma = \dot\sigma$, and
  hence $\gamma = \pi_M \circ \dot\gamma = \pi_M \circ \dot\sigma = \sigma$
  (Remark~\ref{rem:velocity-along}).
\end{proof}

\begin{prop}\label{prop:flow-zero}
  Let $M \in \mathsf{Elast}_{\mathrm{curve}}$ and let $(K, H)$ be a full
  connection on $M$ whose geodesic spray is complete. Then every geodesic flow
  $\Phi$ of $(K, H)$ fixes the zero section, and every geodesic $\gamma$ of
  $K$ satisfies $\gamma(t) = \Phi(t, \dot\gamma(0))$ for all $t \in \R$.
\end{prop}

\begin{proof}
  The existence part of the proof of Theorem~\ref{thm:geodesic-existence}
  applies to any geodesic flow $\Phi$ and shows that, for every $v \in TM$, the
  slice $t \mapsto \Phi(t, v)$ is a geodesic with initial velocity $v$. Given a
  geodesic $\gamma$, the uniqueness part therefore gives
  $\gamma(t) = \Phi(t, \dot\gamma(0))$. For $x \in M$, the constant path $c_x$
  is a geodesic with initial velocity $0_M(x)$: indeed $\dot c_x = 0_M \circ c_x$
  is constant, so $T(\dot c_x) \circ \partial = 0_{TM} \circ \dot c_x$ and
  $\nabla^K_\partial \dot c_x = K \circ 0_{TM} \circ \dot c_x = 0_M \circ c_x$,
  since $K$ preserves zero sections (Remark~\ref{rem:K-linear-bundle}). Hence
  $\Phi(t, 0_M(x)) = c_x(t) = x$.
\end{proof}

\subsection{The energy functional and the first variation formula}%
\label{subsec:energy}

Let $(M, g)$ be an elastic Riemannian diffeological space.
We define the \textbf{energy functional}
$E \colon \Path(M; x, y) \to \R$ by
\begin{align}\label{eq:energy}
  E(\gamma)
    = \frac{1}{2}\int_0^1 g([\gamma, t, \partial], [\gamma, t, \partial]) \, dt
    = \frac{1}{2}\int_0^1 g(\gamma)(t)(\partial, \partial) \, dt,
\end{align}
in the notation of Notation~\ref{not:g-along-plot}; the value $E(\gamma)$ is
finite since $[0,1]$ is compact and the integrand is smooth in $t$.

\begin{defn}[Variation plot]\label{def:variation-plot}
  Let $\gamma \in \Path(M; x, y)$.
  A $1$-plot $P \colon (-\varepsilon, \varepsilon) \to \Path(M; x, y)$
  is called a \textbf{variation plot} of $\gamma$ if it satisfies
  $P(0) = \gamma$, and $P(s)(0) = x$ and $P(s)(1) = y$ for all
  $s \in (-\varepsilon, \varepsilon)$.
\end{defn}

\begin{defn}[Critical path]\label{def:critical-path}
  A path $\gamma \in \Path(M; x, y)$ is called a \textbf{critical path}
  if, for every variation plot $P$ of $\gamma$,
  \begin{align*}
    \left. \frac{d(E \circ P)}{ds} \right|_{s = 0} = 0.
  \end{align*}
\end{defn}

We first record a form of metric compatibility for covariant derivatives along
maps (Definition~\ref{def:covariant-derivative-along}), which is what the
variation computation actually requires: the fields $V_s, V_t$ below are
vector fields along $\ad(P)$, not vector fields on $M$.

\begin{prop}[Metric compatibility along maps]\label{prop:metric-compat-along}
  Let $(M, g)$ be an elastic Riemannian diffeological space with Levi-Civita
  connection $K$, let $f \colon N \to M$ be a smooth map of elastic
  diffeological spaces, and let $V, W$ be vector fields along $f$
  (Definition~\ref{def:vf-along-map}). Then for every $X \in \mathfrak{X}(N)$,
  \begin{align*}
    X\bigl(g(V, W)\bigr)
      = g\bigl(\nabla^K_X V,\, W\bigr) + g\bigl(V,\, \nabla^K_X W\bigr),
  \end{align*}
  where $g(V, W) \colon N \to \R$ is $p \mapsto g(V(p), W(p))$ and
  $\nabla^K_X V := K \circ T(V) \circ X$.
\end{prop}

\begin{proof}
  Since $\pi_M \circ V = f = \pi_M \circ W$, the pair $\langle V, W \rangle$ is
  a smooth map $N \to T_2M = TM \times_M TM$, and $g(V, W) = g \circ \langle V,
  W \rangle$. As $T$ preserves the fibre product $T_2M$, the pair $(T(V) \circ
  X, T(W) \circ X)$ equals $T(\langle V, W \rangle) \circ X$, a map
  $N \to T(T_2M)$. Composing the metric compatibility identity
  \eqref{eq:metric-compatibility} with $T(\langle V, W\rangle) \circ X$ and
  arguing exactly as in Proposition~\ref{prop:metric-compatibility-vf} --- using
  $\pi_{TM} \circ T(V) \circ X = V \circ \pi_N \circ X = V$ and likewise for $W$
  --- gives
  \begin{align*}
    d\bigl(g(V,W)\bigr) \circ X
      = g(K \circ T(V) \circ X,\, W) + g(V,\, K \circ T(W) \circ X),
  \end{align*}
  which is the assertion, since $d(g(V,W)) \circ X = X(g(V,W))$
  (Proposition~\ref{prop:pr2-Tf}, Definition~\ref{def:vf-action}).
\end{proof}

The proof of the first variation formula uses the following symmetry lemma.

\begin{lem}\label{lem:symmetry}
  Let $K$ be a torsion-free vertical connection on $M$.
  For any plot $P \colon U_P \subseteq \R^2 \to M$,
  \begin{align*}
    \nabla^{K}_{\partial_s}[P, (s, t), \partial_t]
      = \nabla^{K}_{\partial_t}[P, (s, t), \partial_s].
  \end{align*}
\end{lem}

\begin{proof}
  By the definition of the covariant derivative along $P$,
  \begin{align*}
    \nabla^{K}_{\partial_s} [P, (s, t), \partial_t]
      = K \circ T^2(P) \circ T(\partial_t) \circ \partial_s,
    \qquad
    \nabla^{K}_{\partial_t}[P, (s, t), \partial_s]
      = K \circ T^2(P) \circ T(\partial_s) \circ \partial_t.
  \end{align*}
  In the coordinates of Example~\ref{ex:euclidean-vertical-connection}, the coordinate fields on $U_P \subseteq \R^2$ have constant components, so that $T(\partial_t)(u, v) = (u, e_2, v, 0)$, and hence
  \begin{align*}
    T(\partial_t) \circ \partial_s(u)
      = (u, e_2, e_1, 0)
      = \tau_{\R^2}(u, e_1, e_2, 0)
      = \tau_{\R^2} \circ T(\partial_s) \circ \partial_t(u),
  \end{align*}
  an instance of the classical relation between the canonical flip and the Lie bracket of vector fields (cf.\ \cite[Lemma~6.13]{KMS}).
  By naturality of $\tau$, $T^2(P) \circ \tau_{\R^2} = \tau_M \circ T^2(P)$, and since $K$ is torsion-free ($K \circ \tau_M = K$),
  \begin{align*}
    \nabla^{K}_{\partial_s}[P, (s, t), \partial_t]
      &= K \circ \tau_M \circ T^2(P) \circ T(\partial_s) \circ \partial_t \\
      &= K \circ T^2(P) \circ T(\partial_s) \circ \partial_t
       = \nabla^{K}_{\partial_t}[P, (s, t), \partial_s]. \qedhere
  \end{align*}
\end{proof}

\begin{lem}[First Variation Formula]\label{lem:first-variation-formula}
  Let $(M, g)$ be an elastic Riemannian diffeological space, and let $K$ be a Levi-Civita connection for $g$. 
  For $\gamma \in \Path(M; x, y)$ and any variation plot $P$ of $\gamma$, consider the vector fields along $\ad(P)$,
  \begin{align*}
    V_s := [\ad(P), (s, t), \partial_s],
    \qquad
    V_t := [\ad(P), (s, t), \partial_t].
  \end{align*}
  Then
  \begin{align*}
      \left. \frac{d(E \circ P)}{ds} \right|_{s = 0}
        = - \int_0^1 g\bigl(V_s,\, \nabla^{K}_{\partial_t} V_t\bigr)
          \Big|_{s=0} \, dt.
  \end{align*}
\end{lem}

\begin{rem}\label{rem:variation-velocity}
For any variation plot $P$, $P(s)$ is a path on $M$ and hence a plot.
For the inclusion $\iota_s \colon \R \to \R^2$, $t \mapsto (s, t)$,
we have $\ad(P) \circ \iota_s = P(s)$. By the colimit description of the
tangent space (Remark~\ref{rem:tangent-structure-2}),
$\bigl[P(s), t, \tfrac{d}{dt}\bigr] = \bigl[\ad(P), (s,t), \partial_t\bigr]$,
so $V_t|_{s}$ is the velocity field of the path $P(s)$.
\end{rem}

\begin{proof}[Proof of Lemma~\ref{lem:first-variation-formula}]
  By \eqref{eq:energy} and differentiation under the integral sign,
  \begin{align*}
      \frac{d(E \circ P)}{ds}
        = \frac{1}{2} \int_0^1 \frac{\partial}{\partial s} g(V_t, V_t) \, dt.
  \end{align*}
  Since $K$ is a Levi-Civita connection and $V_t$ is a vector field along
  $\ad(P)$, metric compatibility along maps
  (Proposition~\ref{prop:metric-compat-along}, with $f = \ad(P)$ and
  $X = \partial_s$) gives
  $\frac{\partial}{\partial s} g(V_t, V_t) = 2g(\nabla^{K}_{\partial_s} V_t, V_t)$,
  hence
  \begin{align*}
      \frac{d(E \circ P)}{ds}
        = \int_0^1 g(\nabla^{K}_{\partial_s} V_t, V_t) \, dt.
  \end{align*}
  By Lemma~\ref{lem:symmetry},
  $\nabla^{K}_{\partial_s} V_t = \nabla^{K}_{\partial_t} V_s$, so applying
  Proposition~\ref{prop:metric-compat-along} once more (now with
  $X = \partial_t$),
  $g(\nabla^{K}_{\partial_t} V_s, V_t)
  = \frac{\partial}{\partial t} g(V_s, V_t)
  - g(V_s, \nabla^{K}_{\partial_t} V_t)$, we obtain
  \begin{align*}
      \frac{d(E \circ P)}{ds}
        = \left[ g(V_s, V_t) \right]_{t=0}^{t=1}
           - \int_0^1 g\bigl(V_s,\, \nabla^{K}_{\partial_t} V_t\bigr) \, dt.
  \end{align*}
  Since the variation plot is constant at the endpoints,
  $V_s|_{t=0} = V_s|_{t=1} = 0$, so the boundary term vanishes.
  Evaluating at $s = 0$ gives the formula.
\end{proof}

\begin{prop}\label{prop:geodesic-critical-path}
  Let $(M, g)$ be as in Lemma~\ref{lem:first-variation-formula}.
  If $\gamma \in \Path(M; x, y)$ satisfies the geodesic equation on $[0,1]$
  for the Levi-Civita connection $K$ (in particular, if $\gamma$ is a
  geodesic), then $\gamma$ is a critical path of the energy functional.
\end{prop}

\begin{proof}
  For any variation plot $P$ with $P(0) = \gamma$, the field $V_t|_{s=0}$ is the
  velocity $\dot\gamma$ (Remark~\ref{rem:variation-velocity}), so
  $\nabla^{K}_{\partial_t} V_t|_{s=0} = \nabla^K_\partial \dot\gamma$
  vanishes on $[0,1]$ by assumption. The first variation formula
  (Lemma~\ref{lem:first-variation-formula}) then gives
  $\frac{d(E \circ P)}{ds}\big|_{s=0} = 0$, so $\gamma$ is a critical path.
\end{proof}

\subsection{Critical paths are geodesics}\label{subsec:critical}

We now prove the converse of Proposition~\ref{prop:geodesic-critical-path}.
Throughout this subsection and the next we make the following standing
assumptions: $(M, g)$ is an elastic Riemannian diffeological space;
$(K, H)$ is a full connection on $M$ whose vertical part $K$ is a Levi-Civita
connection for $g$; and $\Phi \colon \R \times TM \to M$ is a geodesic flow of
$(K, H)$ that fixes the zero section (Definition~\ref{def:fix-zero}). By
Proposition~\ref{prop:flow-zero}, these assumptions hold whenever
$M \in \mathsf{Elast}_{\mathrm{curve}}$, $K$ is a Levi-Civita connection for
$g$, and the geodesic spray of $(K, H)$ is complete.

\begin{lem}\label{lem:flow-initial}
  For every $v \in TM$,
  \begin{align*}
    \mathrm{(i)}\quad \Phi(0, v) = \pi_M(v),
    \qquad\qquad
    \mathrm{(ii)}\quad \Phi^{(1)}(0, v) = v,
  \end{align*}
  where
  $\Phi^{(1)} = T(\Phi) \circ (\partial \times 0_{TM}) \colon \R \times TM \to TM$
  (Definition~\ref{def:nth-derivative}). In particular, for each $v$ the curve
  $s \mapsto \Phi(s, v)$ is a geodesic with base point $\pi_M(v)$ and initial
  velocity $v$ (proof of Theorem~\ref{thm:geodesic-existence}).
\end{lem}

\begin{proof}
  The initial condition of the second-order dynamical system
  $(M, S, \id_{TM})$ (Definition~\ref{def:nth-solution}) is exactly
  $\Phi^{(1)}(0, v) = v$, which is (ii). For (i),
  naturality of $\pi$ gives
  $\pi_M \circ T(\Phi) = \Phi \circ \pi_{\R \times TM}$, and since $T$ preserves
  products (Proposition~\ref{prop:tangent-structure-1}),
  $\pi_{\R \times TM} = \pi_\R \times \pi_{TM}$. Hence
  \begin{align*}
    \pi_M \circ \Phi^{(1)}
      = \Phi \circ (\pi_\R \times \pi_{TM}) \circ (\partial \times 0_{TM})
      = \Phi,
  \end{align*}
  using $\pi_\R \circ \partial = \id_\R$ and
  $\pi_{TM} \circ 0_{TM} = \id_{TM}$. Evaluating at $(0, v)$ and
  substituting (ii) gives $\Phi(0, v) = \pi_M(\Phi^{(1)}(0, v)) = \pi_M(v)$.
\end{proof}

\begin{lem}[Realization of variation fields]\label{lem:realization}
  Let $\gamma \in \Path(M; x, y)$, and let $W \colon \R \to TM$ be a smooth
  vector field along $\gamma$ (so $\pi_M \circ W = \gamma$) with
  $W(0) = 0_M(x)$ and $W(1) = 0_M(y)$. Then
  \begin{align*}
    \ad(P_W)(s, t) := \Phi\bigl(s, W(t)\bigr)
  \end{align*}
  defines a variation plot $P_W$ of $\gamma$
  (Definition~\ref{def:variation-plot}), and its variation field in the
  $s$-direction satisfies
  \begin{align*}
    \left. V_s \right|_{s=0} = [\ad(P_W), (0, t), \partial_s] = W.
  \end{align*}
\end{lem}

\begin{proof}
  The map $(s, t) \mapsto \Phi(s, W(t))$ is smooth, being a composite of the
  smooth maps $\Phi$ and $W$; by cartesian closedness of $\mathsf{Diff}$ it
  corresponds to a $1$-plot $P_W$ into the path space. By
  Lemma~\ref{lem:flow-initial}\,(i),
  $\ad(P_W)(0, t) = \Phi(0, W(t)) = \pi_M(W(t)) = \gamma(t)$, so $P_W(0) = \gamma$.
  Moreover, since $\Phi$ fixes the zero section, $\ad(P_W)(s, 0) = \Phi(s, 0_M(x)) = x$ for all $s$, and
  likewise $\ad(P_W)(s, 1) = y$. Thus $P_W$ is a variation plot in the sense of
  Definition~\ref{def:variation-plot}.

  Finally, for each fixed $t$ the curve
  $s \mapsto \ad(P_W)(s, t) = \Phi(s, W(t))$ is a geodesic with initial
  velocity $W(t)$ (Lemma~\ref{lem:flow-initial}), so, writing
  $\iota_t \colon \R \to \R^2$, $s \mapsto (s, t)$, for the slice map in the
  $s$-direction, Lemma~\ref{lem:flow-initial}\,(ii) gives
  \begin{align*}
   [\ad(P_W),(0,t),\partial_s]
     &= T(\ad(P_W))\circ T(\iota_t)\circ\partial\big|_0
      = T(\Phi)\circ(\partial\times 0_{TM})\,(0,\,W(t)) \\
     &= \Phi^{(1)}(0,\,W(t))
      = W(t).
  \end{align*}
  That is, $V_s|_{s=0} = W$.
\end{proof}

\begin{prop}[Converse of Proposition~\ref{prop:geodesic-critical-path}]%
\label{prop:critical-geodesic}
  If $\gamma \in \Path(M; x, y)$ is a critical path of the energy functional
  (Definition~\ref{def:critical-path}), then $\gamma$ satisfies the geodesic
  equation on $[0,1]$.
\end{prop}

\begin{proof}
  Set $Z := \nabla^{K}_\partial \dot\gamma = K \circ T(\dot\gamma) \circ \partial$,
  a vector field along $\gamma$ (Definition~\ref{def:covariant-derivative-along},
  Remark~\ref{rem:velocity-along}). We must show that $Z(t) = 0_M(\gamma(t))$
  for $t \in [0,1]$.

  Let $W$ be any vector field along $\gamma$ vanishing at the endpoints. Taking
  the variation plot $P_W$ of Lemma~\ref{lem:realization}, we have
  $V_s|_{s=0} = W$; and since $P_W(0) = \gamma$,
  $V_t|_{s=0} = \dot\gamma$ (Remark~\ref{rem:variation-velocity}), whence
  $\nabla^{K}_{\partial_t} V_t|_{s=0} = \nabla^{K}_\partial \dot\gamma = Z$.
  The first variation formula (Lemma~\ref{lem:first-variation-formula}) and
  criticality of $\gamma$ give
  \begin{align}\label{eq:converse-star}
    \int_0^1 g(W, Z)\, dt
      = - \left. \frac{d(E \circ P_W)}{ds} \right|_{s=0}
      = 0
    \qquad \text{for every such } W.
  \end{align}
  The conclusion now follows from the fundamental lemma of the calculus of
  variations, whose proof is identical to the manifold case once the test
  field is produced within the elastic setting. Let $h := g(Z, Z)$, a smooth
  non-negative function on $\R$. If $h(t_0) > 0$ for some $t_0 \in (0,1)$,
  then $h > 0$ on an interval $(a,b) \ni t_0$ with $(a,b) \subseteq (0,1)$;
  choosing
  $\varphi \in C^\infty(\R,\R)$ with $\operatorname{supp}\varphi \subset (a,b)$,
  $\varphi \ge 0$, $\varphi(t_0) > 0$ and setting $W := \varphi Z := \kappa_M
  \circ \langle \varphi, Z\rangle$ yields a vector field along $\gamma$
  vanishing at the endpoints (as $\varphi(0) = \varphi(1) = 0$), for which
  \eqref{eq:converse-star} and bilinearity give
  $0 = \int_0^1 \varphi\, h \, dt > 0$, a contradiction. Hence $h = 0$ on
  $(0,1)$, and by continuity on $[0,1]$; definiteness of $g$ gives
  $Z(t) = 0_M(\gamma(t))$ for $t \in [0,1]$.
\end{proof}


\subsection{Length-minimizing paths}\label{subsec:length}

\begin{defn}\label{def:length}
  The \textbf{length} of a path $\gamma \in \Path(M; x, y)$ is
  \begin{align*}
    L(\gamma)
      = \int_0^1 \bigl( g(\gamma)(t)(\partial, \partial) \bigr)^{1/2}\, dt,
  \end{align*}
  as in \cite{KSS}, in the notation of Notation~\ref{not:g-along-plot}; the
  integrand is continuous. The path $\gamma$ is \textbf{regular} if
  $g(\gamma)(t)(\partial, \partial) > 0$ for all $t \in [0,1]$.
\end{defn}

\begin{lem}\label{lem:energy-length}
  For every $\gamma \in \Path(M; x, y)$,
  \begin{align*}
    L(\gamma)^2 \le 2\, E(\gamma),
  \end{align*}
  with equality if and only if
  $t \mapsto g(\gamma)(t)(\partial, \partial)$ is constant on $[0,1]$.
\end{lem}

\begin{proof}
  Write $f(t) := (g(\gamma)(t)(\partial, \partial))^{1/2} \ge 0$.
  By the Cauchy--Schwarz inequality,
  \begin{align*}
    L(\gamma) = \int_0^1 f \cdot 1 \, dt
      \le \Bigl( \int_0^1 f^2 \, dt \Bigr)^{1/2}
          \Bigl( \int_0^1 1 \, dt \Bigr)^{1/2}
      = \bigl( 2 E(\gamma) \bigr)^{1/2}.
  \end{align*}
  Squaring gives the inequality; since $f$ is continuous, equality holds if
  and only if $f$ is constant on $[0,1]$.
\end{proof}

\begin{lem}\label{lem:reparam}
  Let $\gamma \in \Path(M; x, y)$ be regular. Then there is a diffeomorphism
  $\psi \colon \R \to \R$ with $\psi(0) = 0$ and $\psi(1) = 1$ such that the
  path $\widetilde\gamma := \gamma \circ \psi^{-1} \in \Path(M; x, y)$ satisfies
  \begin{align*}
    g(\widetilde\gamma)(u)(\partial, \partial) = L(\gamma)^2
    \quad (u \in [0,1]),
    \qquad
    L(\widetilde\gamma) = L(\gamma),
    \qquad
    E(\widetilde\gamma) = \tfrac{1}{2} L(\gamma)^2.
  \end{align*}
\end{lem}

\begin{proof}
  Write $e(\tau) := g(\gamma)(\tau)(\partial, \partial)$, a smooth
  non-negative function on $\R$. By regularity and compactness of $[0,1]$,
  $e > 0$ on an open interval $J \supseteq [0,1]$, so $f := e^{1/2}$ is smooth
  and positive on $J$, and $L := L(\gamma) = \int_0^1 f\,d\tau > 0$. Choose
  $\chi \in C^\infty(\R, [0,1])$ with compact support in $J$ and $\chi = 1$ on
  a neighbourhood of $[0,1]$, and put
  \begin{align*}
    \psi(t) := L^{-1} \int_0^t \bigl( \chi f + (1 - \chi)\, L \bigr)\, d\tau ,
  \end{align*}
  where $\chi f$ is extended by $0$ outside $J$. The integrand is smooth and
  positive on $\R$ and equal to $L$ outside a compact set, so $\psi$ is a
  diffeomorphism of $\R$; moreover $\psi(0) = 0$, $\psi' = L^{-1} f$ on $[0,1]$,
  and $\psi(1) = L^{-1}\int_0^1 f\, d\tau = 1$. Thus $\psi$ maps $[0,1]$ onto
  itself, $\widetilde\gamma \in \Path(M; x, y)$, and
  $(\psi^{-1})'(u) = L / f(\psi^{-1}(u))$ for $u \in [0,1]$.

  Since $\psi^{-1}$ is a morphism of plots from $\widetilde\gamma$ to $\gamma$
  ($\gamma \circ \psi^{-1} = \widetilde\gamma$), functoriality of $T$ gives
  $T(\widetilde\gamma) = T(\gamma) \circ T(\psi^{-1})$, so by
  Remark~\ref{rem:tangent-structure-2}, for each $u$,
  \begin{align*}
    [\widetilde\gamma, u, \partial]
      = [\gamma,\ \psi^{-1}(u),\ (\psi^{-1})'(u)]
      = \kappa_M\bigl( (\psi^{-1})'(u),\ [\gamma, \psi^{-1}(u), \partial] \bigr),
  \end{align*}
  the second equality by
  $\kappa_M(r, [P, u, u']) = [P, u, r u']$
  (Definition~\ref{def:nat-trans}). By bilinearity of $g$, for $u \in [0,1]$,
  \begin{align*}
    g(\widetilde\gamma)(u)(\partial, \partial)
      = (\psi^{-1})'(u)^2 \, g(\gamma)(\psi^{-1}(u))(\partial, \partial)
      = \Bigl( \frac{L}{f(\psi^{-1}(u))} \Bigr)^{2} f(\psi^{-1}(u))^2
      = L^2 .
  \end{align*}
  Hence $L(\widetilde\gamma) = \int_0^1 L\, du = L$ and
  $E(\widetilde\gamma) = \tfrac12 \int_0^1 L^2\, du = \tfrac12 L^2$.
\end{proof}

\begin{prop}\label{prop:minimizer-geodesic}
  Assume the standing hypotheses of Section~\ref{subsec:critical}. Let
  $\gamma \in \Path(M; x, y)$ be a regular path that minimizes length, i.e.\
  $L(\gamma) \le L(\sigma)$ for all $\sigma \in \Path(M; x, y)$. Then its
  constant-speed reparametrization $\widetilde\gamma$ (Lemma~\ref{lem:reparam})
  satisfies the geodesic equation on $[0,1]$.
\end{prop}

\begin{proof}
  For any $\sigma \in \Path(M; x, y)$, Lemma~\ref{lem:energy-length} and
  minimality of $L(\gamma)$ give
  \begin{align*}
    E(\sigma) \ge \tfrac12 L(\sigma)^2 \ge \tfrac12 L(\gamma)^2
      = E(\widetilde\gamma),
  \end{align*}
  the last equality by Lemma~\ref{lem:reparam}. Hence $\widetilde\gamma$
  minimizes $E$ over $\Path(M; x, y)$. Consequently, for every variation plot
  $P$ of $\widetilde\gamma$ the smooth function $s \mapsto E(P(s))$ attains a
  minimum at $s = 0$, so $\frac{d(E \circ P)}{ds}\big|_{s=0} = 0$; that is,
  $\widetilde\gamma$ is a critical path. By
  Proposition~\ref{prop:critical-geodesic}, $\widetilde\gamma$ satisfies the
  geodesic equation on $[0,1]$.
\end{proof}

\subsection{Proofs of Theorems~\ref{mainthm:A} and~\ref{mainthm:B}}\label{subsec:proof-main}

Under the hypotheses of Theorem~\ref{mainthm:A}, the standing assumptions of
Section~\ref{subsec:critical} hold. For~(i), a path satisfying the geodesic
equation on $[0,1]$ is a critical path by
Proposition~\ref{prop:geodesic-critical-path}, and conversely a critical path
satisfies the geodesic equation on $[0,1]$ by
Proposition~\ref{prop:critical-geodesic}. Part~(ii) is
Proposition~\ref{prop:minimizer-geodesic}.

Under the hypotheses of Theorem~\ref{mainthm:B}, a geodesic flow exists by
Definition~\ref{def:geodesic-spray}; existence and uniqueness of geodesics is
Theorem~\ref{thm:geodesic-existence}; and the remaining assertions are
Proposition~\ref{prop:flow-zero}. \qed

\section{Examples}\label{sec:examples}

\subsection{Quotients and irrational tori}\label{subsec:irrational-torus}

The descent results of~\cite{ShioQ} provide non-manifold spaces to which
Theorems~\ref{mainthm:A} and~\ref{mainthm:B} apply. For a Q-chart
$\phi\colon X \to Y$ in the sense of Miyamoto~\cite{M}, a class of subductions
containing the quotient maps of free actions of countable groups, the iterated
tangent bundles of $X$ are the pullbacks of those of $Y$; consequently
invariant metrics and connections on $X$ descend to $Y$, preserving the
Levi-Civita property \cite[Theorem~3.1]{ShioQ}. For principal actions of
diffeologically discrete groups, membership in
$\mathsf{Elast}_{\mathrm{curve}}$ and the completeness of the geodesic spray
descend as well \cite[Theorem~4.8, Corollary~4.10]{ShioQ}.

\begin{prop}\label{prop:quotient-main}
  Let $(X, g)$ be a geodesically complete finite-dimensional Riemannian
  manifold on which a countable group $\Lambda$ acts freely by isometries.
  Then $Y = X/\Lambda$, with the quotient diffeology, belongs to
  $\mathsf{Elast}_{\mathrm{curve}}$, and $g$ and its Levi-Civita connection
  descend to a Riemannian metric $g_Y$ and a full connection on $Y$ whose
  vertical part is a Levi-Civita connection for $g_Y$ and whose geodesic
  spray is complete. In particular, Theorems~\ref{mainthm:A}
  and~\ref{mainthm:B} apply to $Y$.
\end{prop}

\begin{proof}
  The action is principal \cite[Example~4.3(b)]{ShioQ}. The metric $g$ and
  its Levi-Civita connection $K$ are invariant under the isometries in
  $\Lambda$, and so is the horizontal connection $H$, being uniquely determined
  by the effective vertical connection $K$
  (Proposition~\ref{prop:effective-iff-full}); the geodesic spray of $(K, H)$
  is complete (Example~\ref{ex:riemannian-manifold}). Since
  $X \in \mathsf{Elast}_{\mathrm{curve}}$ (Remark~\ref{rem:manifolds-curve}),
  the assertions follow from \cite[Theorem~3.1, Corollary~4.10]{ShioQ}.
\end{proof}

\begin{ex}[Irrational tori]\label{ex:torus-critical}
  Let $\Tor_\alpha = \R/\Lambda_\alpha$, $\Lambda_\alpha := \Z + \alpha\Z$,
  $\alpha \notin \Q$, and $[x] := x + \Lambda_\alpha$. Its D-topology is
  indiscrete and the pseudo-distance induced by any Riemannian metric
  vanishes identically \cite[Lemma~5.1, Proposition~5.7]{ShioQ}.
  Proposition~\ref{prop:quotient-main} applies to the translation action of
  $\Lambda_\alpha$ on $\R$; moreover all Riemannian metrics on $\Tor_\alpha$
  have the same Levi-Civita connection, whose geodesics are the paths
  $t \mapsto [x + tv]$ \cite[Theorem~5.6]{ShioQ}. Hence, for
  $[x], [y] \in \Tor_\alpha$, the paths
  \begin{align*}
    \gamma_\eta(t) := [\,x + t(y - x + \eta)\,], \qquad \eta \in \Lambda_\alpha,
  \end{align*}
  are geodesics in $\Path(\Tor_\alpha; [x], [y])$, and thus critical paths of
  the energy (Proposition~\ref{prop:geodesic-critical-path}). They are
  pairwise distinct, since their initial velocities $y - x + \eta$ differ, and
  for the metric induced by $\rho\,dt^2$ their lengths
  $\sqrt{\rho}\,|y - x + \eta|$ accumulate at $0$ by density of
  $\Lambda_\alpha$. Thus the energy has infinitely many critical points on
  $\Path(\Tor_\alpha; [x], [y])$, although the pseudo-distance does not
  distinguish points; for $[x] \neq [y]$ no path minimizes length
  \cite[Proposition~5.7]{ShioQ}, so part~(ii) of Theorem~\ref{mainthm:A} is
  vacuous here (compare Example~\ref{ex:mapping-flat}).
\end{ex}

\subsection{Mapping spaces: the tangent structure}\label{sec:mapping-space}

We now turn to mapping spaces. For elastic diffeological spaces, the explicit
descriptions of geometric objects in terms of domains that local coordinates
provide in Riemannian geometry are generally not available: in general
$\colim_{P \in \D(M)}\hat{T}^2U_P \neq T^2M$, and vector fields cannot be
expressed as limits or colimits of vector fields on domains. On mapping spaces,
however, the tangent structure, connections and curvature, the Levi-Civita
connection, and geodesics all lift pointwise from the target.

The following results are due to Blohmann.

\begin{prop}[{\cite[Proposition~5.8]{B23}}]\label{prop:mapping-space-tangent-stable}
  Let $V$ be a tangent-stable elastic vector space.
  For any diffeological space $Z$, $C^\infty(Z, V)$ is elastic and
  \begin{align*}
    TC^\infty(Z, V) \cong C^\infty(Z, TV).
  \end{align*}
\end{prop}

\begin{prop}[{\cite[Corollary~5.9]{B23}}]\label{prop:mapping-space-manifolds}
  The mapping space $C^\infty(M, N)$ between manifolds is elastic, and
  \begin{align*}
    TC^\infty(M, N) \cong C^\infty(M, TN).
  \end{align*}
\end{prop}

In both cases the isomorphism is the canonical comparison map, which exists
for arbitrary $M$ and $N$: for $m \in M$, let
$\evl_m \colon C^\infty(M, N) \to N$ denote the evaluation, and define
\begin{align*}
  \Psi_{M, N} \colon TC^\infty(M, N) \longrightarrow C^\infty(M, TN),
  \qquad
  \Psi_{M, N}(v)(m) := T(\evl_m)(v),
\end{align*}
which is smooth by cartesian closedness of $\mathsf{Diff}$.
This motivates the following definition.

\begin{defn}\label{def:tangent-commuting}
A pair $(M, N)$ of diffeological spaces with $N$ elastic is called
\textbf{tangent-commuting} if $C^\infty(M, N)$ is elastic and the canonical map
$\Psi_{M, N} \colon TC^\infty(M, N) \to C^\infty(M, TN)$ is a diffeomorphism.
\end{defn}

Under this identification, $T(\evl_m)$ corresponds to the evaluation of
$C^\infty(M, TN)$ at $m$; iterating,
\begin{align*}
   T^k(\evl_m) \colon T^kC^\infty(M, N) \to T^kN
\end{align*}
corresponds to the evaluation of $C^\infty(M, T^kN)$ at $m$,
provided the pairs $(M, T^jN)$ for $j < k$ are tangent-commuting.

\begin{convention}\label{conv:hereditarily-tc}
  Throughout the rest of this section we assume that $(M, T^kN)$ is tangent-commuting for every $k \geq 0$ (where $T^0N = N$). When $N$ is a manifold this follows from Proposition~\ref{prop:mapping-space-manifolds}, since $T^kN$ is again a manifold; when $N$ is a tangent-stable elastic vector space it follows from Proposition~\ref{prop:mapping-space-tangent-stable}, since $TN \cong N \times N$ is again a tangent-stable elastic vector space.
  We write $\X := C^\infty(M, N)$ and suppress the identifications $T^k\X \cong C^\infty(M, T^kN)$ from the notation. Note also that, being a right adjoint by cartesian closedness, $C^\infty(M, -)$ preserves fibre products; in particular $T_2\X \cong C^\infty(M, T_2N)$ and $T\X \times_{\X} T\X \times_{\X} T\X \cong C^\infty(M, TN \times_N TN \times_N TN)$.
\end{convention}

This assumption is used repeatedly to identify
$T^kC^\infty(M,N)$
with
$C^\infty(M,T^kN)$,
thereby allowing all higher-order constructions (connections, curvature and geodesic sprays) to be evaluated pointwise.

Since the tangent structure $(\pi_N, 0_N, \kappa_N, +_N, \tau_N, \lambda_N)$ on $N$ consists of natural transformations, and post-composition $C^\infty(M, -)$ is a covariant functor, the tangent structure on $\X$ is obtained by applying $C^\infty(M, -)$ to that of $N$. Explicitly:
\begin{description}
  \item[Projection]
    $\pi_{\X}(v) = \pi_N \circ v$,\quad
    $v \in C^\infty(M, TN)$.
  \item[Zero section]
    $0_{\X}(f) = 0_N \circ f$,\quad
    $f \in C^\infty(M, N)$.
  \item[Addition]
    $(v_1 +_{\X} v_2)(x) = v_1(x) +_N v_2(x)$,\quad
    $x \in M$.
  \item[Scalar multiplication]
    $\kappa_{\X}(r, v)(x) = \kappa_N(r, v(x))$,\quad
    $r \in \R$.
  \item[Symmetric structure]
    $\tau_{\X}(u) = \tau_N \circ u$,\quad
    $u \in C^\infty(M, T^2N)$.
  \item[Vertical lift]
    $\lambda_{\X}(u) = \lambda_N \circ u$.
\end{description}
The axioms of a tangent structure are equations between composites of these
natural transformations, and are therefore preserved by the functor
$C^\infty(M, -)$.

\subsection{Mapping spaces: connections and curvature}%
\label{subsec:mapping-connection}

Connections on mapping spaces likewise lift from the target space.

\begin{prop}\label{prop:mapping-space-connection}
 Let $K_N$ and $H_N$ be a vertical connection and a horizontal connection on
 $N$, respectively. Define
 $K_{\X} \colon T^2\X \to T\X$ and
 $H_{\X} \colon T\X \times_{\X} T\X \to T^2\X$ by
 \begin{align*}
   K_{\X}(u) &:= K_N \circ u,
     \qquad u \in C^\infty(M, T^2N),\\
   H_{\X}(v, w) &:= H_N \circ \langle v, w \rangle,
     \qquad \langle v, w \rangle \in C^\infty(M, TN \times_N TN).
 \end{align*}
 Then $K_{\X}$ is a vertical connection on $\X$ and $H_{\X}$ is a horizontal
 connection on $\X$. If $K_N$ is torsion-free, so is $K_{\X}$; if $(K_N, H_N)$
 is a full connection, so is $(K_{\X}, H_{\X})$.
\end{prop}

\begin{proof}
  Under the identifications of Convention~\ref{conv:hereditarily-tc}, $K_{\X} = C^\infty(M, K_N)$, $H_{\X} = C^\infty(M, H_N)$, and every structure map appearing in Definitions~\ref{def:vertical-connection}, \ref{def:horizontal-connection} and~\ref{def:full-connection} is the image under $C^\infty(M, -)$ of the corresponding map for $N$. 
  Each axiom, as well as torsion-freeness, is an equation between composites of such maps,
  and is therefore inherited from $N$ by functoriality; the scalar axioms, which involve the parameter $\R$, are likewise verified pointwise in $m \in M$.
\end{proof}

\begin{prop}\label{prop:mapping-space-effective}
   If $K_N$ is an effective vertical connection on $N$ (Definition~\ref{def:effective-vertical-connection}), then $K_{\X}$ is also effective.
\end{prop}

\begin{proof}
  Under the same identifications (using that $C^\infty(M, -)$ preserves
  fibre products),
  \begin{align*}
    \Theta_{\X} = \langle \pi_{T\X}, T(\pi_{\X}), K_{\X} \rangle
      = C^\infty(M, \Theta_N);
  \end{align*}
  since functors preserve isomorphisms,
  $C^\infty(M, \Theta_N^{-1})$ is a smooth inverse of $\Theta_{\X}$, so
  $K_{\X}$ is effective.
\end{proof}

The curvature (Section~\ref{sec:curvature}) also localizes.

\begin{prop}\label{prop:mapping-space-curvature}
  Let $K_N$ be a vertical connection on $N$. Under the identification $T^3\X \cong C^\infty(M, T^3N)$, the curvature morphism of $K_{\X}$ (Definition~\ref{def:curvature-morphism}) is given by post-composition with that of $K_N$:
  \begin{align*}
    C_{K_{\X}}(u) = C_{K_N} \circ u,
    \qquad u \in C^\infty(M, T^3N).
  \end{align*}
  In particular, $K_{\X}$ is flat if and only if $K_N$ is flat.
\end{prop}

\begin{proof}
  Both composites $K_N \circ T(K_N) \circ \tau_{TN}$ and $K_N \circ T(K_N)$, as well as the fibrewise subtraction in the abelian group structure of $\pi_N \colon TN \to N$, are built from structure maps of $N$ that lift by post-composition; in particular the fibrewise addition and inverse of $\pi_{\X}$ are obtained by applying
  $C^\infty(M, -)$ to those of $\pi_N$. 
  Hence $C_{K_{\X}} = C^\infty(M, C_{K_N})$ by the same argument as in Proposition~\ref{prop:mapping-space-connection}.
  If $C_{K_N} = 0$, then $C_{K_{\X}}(u) = C_{K_N} \circ u = 0$ for every $u$.
  Conversely, if $C_{K_{\X}} = 0$, then evaluating at constant maps $u \equiv w$ for $w \in T^3N$ gives $C_{K_N}(w) = 0$.
\end{proof}

\begin{ex}\label{ex:flat-mapping-space}
  For $N = \R^n$ with the canonical vertical connection $K_0$ of Example~\ref{ex:euclidean-vertical-connection}, $C_{K_0} = 0$ by Example~\ref{ex:flat-euclidean}. Hence, for any manifold $M$, the connection $K_{C^\infty(M, \R^n)}$ on the mapping space $C^\infty(M, \R^n)$ is flat.
\end{ex}

\subsection{Mapping spaces: the Levi-Civita connection}%
\label{subsec:mapping-lc}

For the remainder of this section, $M$ is a closed manifold equipped with a
volume form $\vol_M$, and
$(N, g_N)$ is an elastic Riemannian diffeological space.

\begin{defn}\label{def:g-map}
  The \textbf{mapping-space metric} $g_{\mathit{map}}$ on $\X = C^\infty(M, N)$ is defined, for $v, w \in T\X$ lying in a common
  fibre, by
  \begin{align*}
    g_{\mathit{map}}(v, w)
      := \int_M g_N\bigl(T(\evl_m)(v),\, T(\evl_m)(w)\bigr)\, \vol_M(m).
  \end{align*}
  In the plot notation of Notation~\ref{not:g-along-plot} this reads $g_{\mathit{map}}(P)(u)(v, w) = \int_M g_N(\evl_m \circ P)(u)(v, w)\,\vol_M(m)$, as in \cite{KSS}.
\end{defn}

\begin{lem}\label{lem:g-map-metric}
  $g_{\mathit{map}}$ is a Riemannian metric on $\X$ in the sense of
  Definition~\ref{def:metric}; in particular it is definite.
\end{lem}

\begin{proof}
  Write $\mathcal{I} \colon C^\infty(M, \R) \to \R$ for the integration against $\vol_M$, a smooth linear map, and $\mathcal{G} \colon T_2\X \to C^\infty(M, \R)$ for the pointwise metric evaluation, $\mathcal{G}(v, w)(m) := g_N(T(\evl_m)(v), T(\evl_m)(w))$, which is smooth by cartesian closedness. 
  Then $g_{\mathit{map}} = \mathcal{I} \circ \mathcal{G}$ is smooth. 
  Bilinearity and symmetry hold pointwise in $m$ and are preserved by the linear map $\mathcal{I}$; positivity follows since the integrand is non-negative. For definiteness, suppose $g_{\mathit{map}}(v, v) = 0$. 
  The function $m \mapsto g_N(T(\evl_m)(v), T(\evl_m)(v))$ is smooth, non-negative, and has vanishing integral over the closed manifold $M$, hence vanishes identically; definiteness of $g_N$ then gives $T(\evl_m)(v) = 0$ for every $m$, i.e.\ $\Psi_{M,N}(v) = 0$, and so $v = 0$.
\end{proof}

Let $K_N$ be a vertical Levi-Civita connection for $g_N$.

\begin{prop}\label{prop:mapping-space-levi-civita}
   $K_{\X}$ is a vertical Levi-Civita connection for $g_{\mathit{map}}$.
   By Theorem~\ref{thm:uniqueness-cov-deriv}, its covariant derivative $\nabla^{K_{\X}}$ is the unique Levi-Civita covariant derivative for $g_{\mathit{map}}$.
\end{prop}

\begin{proof}
   Torsion-freeness of $K_{\X}$ was shown in Proposition~\ref{prop:mapping-space-connection}. We verify the metric
   compatibility condition \eqref{eq:metric-compatibility} for $g_{\mathit{map}}$.

   For $m \in M$ write
   $h_m := g_N \circ (T(\evl_m) \times_{\evl_m} T(\evl_m))
   \colon T_2\X \to \R$, so that
   $g_{\mathit{map}} = \mathcal{I} \circ \mathcal{G}$ with
   $\mathcal{G}(\xi)(m) = h_m(\xi)$ as in Lemma~\ref{lem:g-map-metric}.
   Since $\mathcal{I}$ is smooth and linear, under the identification
   $TC^\infty(M, \R) \cong C^\infty(M, T\R) \cong C^\infty(M, \R \times \R)$
   its tangent map acts componentwise, whence
   $d(\mathcal{I} \circ \mathcal{G})
    = \mathcal{I} \circ \bigl(m \mapsto d(h_m)\bigr)$; that is,
   \begin{align}\label{eq:diff-under-integral}
     d(g_{\mathit{map}})(W)
       = \int_M d(h_m)(W)\, \vol_M(m),
     \qquad W \in T(T_2\X).
   \end{align}
   Next, $h_m = g_N \circ T_2(\evl_m)$ with
   $T_2(\evl_m) := T(\evl_m) \times T(\evl_m) \colon T_2\X \to T_2N$, so the
   chain rule $d(f \circ h) = d(f) \circ T(h)$ (immediate from
   $d = \pr_2 \circ T$ and functoriality) gives
   $d(h_m) = d(g_N) \circ T(T_2(\evl_m))$, where $T(T_2(\evl_m))$ is
   identified with $T^2(\evl_m) \times T^2(\evl_m)$ on
   $T(T_2\X) \cong T^2\X \times_{T\X} T^2\X$. Substituting the metric
   compatibility condition \eqref{eq:metric-compatibility} for $(N, g_N, K_N)$,
   \begin{align*}
     d(h_m)(W_1, W_2)
       = g_N\bigl(K_N T^2(\evl_m) W_1,\; \pi_{TN} T^2(\evl_m) W_2\bigr)
       + g_N\bigl(\pi_{TN} T^2(\evl_m) W_1,\; K_N T^2(\evl_m) W_2\bigr).
   \end{align*}
   By Convention~\ref{conv:hereditarily-tc}, $T^2(\evl_m)$ is the evaluation
   of $C^\infty(M, T^2N)$ at $m$, so
   $K_N \circ T^2(\evl_m) = T(\evl_m) \circ K_{\X}$ and
   $\pi_{TN} \circ T^2(\evl_m) = T(\evl_m) \circ \pi_{T\X}$. Hence
   \begin{align*}
     d(h_m)(W_1, W_2)
       = h_m\bigl(K_{\X}(W_1),\, \pi_{T\X}(W_2)\bigr)
       + h_m\bigl(\pi_{T\X}(W_1),\, K_{\X}(W_2)\bigr),
   \end{align*}
   and integrating over $M$ via \eqref{eq:diff-under-integral},
   \begin{align*}
     d(g_{\mathit{map}})
       = +_{\R} \circ
         \langle g_{\mathit{map}} \circ (K_{\X} \times_{\pi} \pi_{T\X}),\;
                g_{\mathit{map}} \circ (\pi_{T\X} \times_{\pi} K_{\X}) \rangle,
   \end{align*}
   which is \eqref{eq:metric-compatibility} on $\X$.
\end{proof}

\subsection{Mapping spaces: geodesics}\label{subsec:mapping-geodesics}

We keep the hypotheses of \S\ref{subsec:mapping-lc}. For a path
$\gamma \colon \R \to \X$ and $m \in M$, write
$\gamma_m := \evl_m \circ \gamma \colon \R \to N$, a path on $N$.

\begin{prop}\label{prop:mapping-space-geodesic}
  A path $\gamma \colon \R \to \X$ is a geodesic for $K_{\X}$ if and only if
  $\gamma_m$ is a geodesic on $N$ for $K_N$ for every $m \in M$.
\end{prop}

\begin{proof}
  By naturality of the identifications, for every $m \in M$,
  \begin{align*}
    T(\evl_m)\bigl(\nabla^{K_{\X}}_\partial \dot\gamma\bigr)
      &= T(\evl_m) \circ K_{\X} \circ T(\dot\gamma) \circ \partial
       = K_N \circ T^2(\evl_m) \circ T(\dot\gamma) \circ \partial \\
      &= K_N \circ T\bigl(T(\evl_m) \circ \dot\gamma\bigr) \circ \partial
       = K_N \circ T(\dot\gamma_m) \circ \partial
       = \nabla^{K_N}_\partial \dot\gamma_m,
  \end{align*}
  using $T(\evl_m) \circ \dot\gamma
  = T(\evl_m) \circ T(\gamma) \circ \partial = T(\gamma_m) \circ \partial
  = \dot\gamma_m$, and likewise
  $T(\evl_m) \circ 0_{\X} \circ \gamma = 0_N \circ \gamma_m$.
  Hence, if $\gamma$ is a geodesic, applying $T(\evl_m)$ to
  $\nabla^{K_{\X}}_\partial \dot\gamma = 0_{\X} \circ \gamma$ shows that each
  $\gamma_m$ is a geodesic. Conversely, if every $\gamma_m$ is a geodesic,
  then $T(\evl_m) \circ \nabla^{K_{\X}}_\partial \dot\gamma
  = T(\evl_m) \circ 0_{\X} \circ \gamma$ for every $m$; since the family
  $(T(\evl_m))_{m \in M}$ corresponds under $\Psi_{M,N}$ to the evaluations of
  $C^\infty(M, TN)$, it is jointly injective, so
  $\nabla^{K_{\X}}_\partial \dot\gamma = 0_{\X} \circ \gamma$.
\end{proof}

The geodesic flow, and hence the existence theory of
\S\ref{subsec:geodesic-existence}, also transfers pointwise.

\begin{prop}\label{prop:mapping-space-flow}
  Let $(K_N, H_N)$ be a full
  connection on $N$ whose geodesic spray $S_N$ is complete, with geodesic flow
  $\Phi_N \colon \R \times TN \to N$. Then the geodesic spray $S_{\X}$ of
  $(K_{\X}, H_{\X})$ is complete, with geodesic flow
  \begin{align*}
    \Phi_{\X} \colon \R \times T\X \to \X,
    \qquad
    \Phi_{\X}(s, v)(m) := \Phi_N\bigl(s,\, T(\evl_m)(v)\bigr).
  \end{align*}
  If moreover $N \in \mathsf{Elast}_{\mathrm{curve}}$, then for every
  $v_0 \in T\X$ there is a unique geodesic $\gamma \colon \R \to \X$ with
  $\dot\gamma(0) = v_0$, namely $\gamma = \Phi_{\X} \circ \iota_{v_0}$.
\end{prop}

\begin{proof}
  The map $(s, v, m) \mapsto \Phi_N(s, T(\evl_m)(v))$ is smooth, so
  $\Phi_{\X}$ is smooth by cartesian closedness. Under the identifications,
  $S_{\X} = C^\infty(M, S_N)$ and
  $\evl_m \circ \Phi_{\X} = \Phi_N \circ (\id_\R \times T(\evl_m))$;
  applying $T$ and $T^2$ to the latter and using naturality of $0$, one
  obtains
  \begin{align*}
    T(\evl_m) \circ \Phi_{\X}^{(1)}
      = \Phi_N^{(1)} \circ (\id_\R \times T(\evl_m)),
    \qquad
    T^2(\evl_m) \circ \Phi_{\X}^{(2)}
      = \Phi_N^{(2)} \circ (\id_\R \times T(\evl_m)).
  \end{align*}
  The initial and differential conditions for $\Phi_{\X}$
  (Definition~\ref{def:nth-solution}) then follow from those for $\Phi_N$ by
  the joint injectivity of the families $(T(\evl_m))_m$ and
  $(T^2(\evl_m))_m$, as in Proposition~\ref{prop:mapping-space-geodesic}.
  Hence $S_{\X}$ is complete with geodesic flow $\Phi_{\X}$.

  For the final claim, assume $N \in \mathsf{Elast}_{\mathrm{curve}}$. Existence: as in
  Theorem~\ref{thm:geodesic-existence}, the slice
  $\gamma = \Phi_{\X} \circ \iota_{v_0}$ satisfies
  $\dot\gamma(0) = \Phi_{\X}^{(1)}(0, v_0) = v_0$, and each
  $\gamma_m = \Phi_N(\cdot, T(\evl_m)(v_0))$ is a geodesic on $N$
  (Lemma~\ref{lem:flow-initial}), so $\gamma$ is a geodesic by
  Proposition~\ref{prop:mapping-space-geodesic}. Uniqueness: if $\gamma$ and
  $\sigma$ are geodesics on $\X$ with
  $\dot\gamma(0) = \dot\sigma(0) = v_0$, then for each $m$ the paths
  $\gamma_m$ and $\sigma_m$ are geodesics on $N$ with the same initial
  velocity $T(\evl_m)(v_0)$, so $\gamma_m = \sigma_m$ by
  Theorem~\ref{thm:geodesic-existence} applied to $N$; since the
  evaluations $(\evl_m)_m$ are jointly injective, $\gamma = \sigma$.
\end{proof}

\begin{rem}\label{rem:mapping-space-lc-membership}
  Proposition~\ref{prop:mapping-space-flow} provides existence and uniqueness
  of geodesics on $\X$ without settling whether
  $\X \in \mathsf{Elast}_{\mathrm{curve}}$, i.e.\ whether $\R$ is a curve object
  for $\X$: the uniqueness above is derived pointwise from
  $N \in \mathsf{Elast}_{\mathrm{curve}}$ rather than from a curve-object
  property of $\X$. We do not address here whether
  $N \in \mathsf{Elast}_{\mathrm{curve}}$ implies
  $C^\infty(M, N) \in \mathsf{Elast}_{\mathrm{curve}}$. Nevertheless the
  variational theory of Section~\ref{sec:geodesics} applies to $\X$
  (Theorem~\ref{thm:mapping-variational}): the property of fixing the zero
  section, which is what it requires, is inherited from $N$.
\end{rem}

\begin{thm}\label{thm:mapping-variational}
  In addition to the hypotheses of \S\ref{subsec:mapping-lc}, assume that
  $N \in \mathsf{Elast}_{\mathrm{curve}}$ and that $(K_N, H_N)$ is a full
  connection on $N$ whose vertical part $K_N$ is a Levi-Civita connection for
  $g_N$ and whose geodesic spray is complete. Then the geodesic flow
  $\Phi_\X$ of Proposition~\ref{prop:mapping-space-flow} fixes the zero
  section, and Theorem~\ref{mainthm:A} applies to $(\X, g_{\mathit{map}})$
  and $(K_\X, H_\X)$. Explicitly, for $f_0, f_1 \in \X$:
  \begin{enumerate}[(i)]
    \item a path $\gamma \in \Path(\X; f_0, f_1)$ is a critical path of the
      energy of $g_{\mathit{map}}$ if and only if, for every $m \in M$, the
      path $\gamma_m$ satisfies the geodesic equation of $K_N$ on $[0,1]$;
    \item if $\gamma \in \Path(\X; f_0, f_1)$ is regular and minimizes length,
      then its constant-speed reparametrization satisfies the geodesic
      equation of $K_\X$ on $[0,1]$.
  \end{enumerate}
\end{thm}

\begin{proof}
  By Propositions~\ref{prop:mapping-space-connection}
  and~\ref{prop:mapping-space-levi-civita} and
  Lemma~\ref{lem:g-map-metric}, $g_{\mathit{map}}$ is a Riemannian metric and
  $(K_\X, H_\X)$ is a full connection whose vertical part is a Levi-Civita
  connection for it, and $\Phi_\X$ is a geodesic flow of $(K_\X, H_\X)$ by
  Proposition~\ref{prop:mapping-space-flow}. By naturality of the zero
  section, $T(\evl_m) \circ 0_\X = 0_N \circ \evl_m$, so
  \begin{align*}
    \Phi_\X\bigl(s, 0_\X(f)\bigr)(m)
      = \Phi_N\bigl(s, 0_N(f(m))\bigr)
      = f(m)
  \end{align*}
  for all $s \in \R$, $f \in \X$ and $m \in M$, since $\Phi_N$ fixes the zero
  section by Proposition~\ref{prop:flow-zero}. Hence $\Phi_\X$ fixes the zero
  section, and Theorem~\ref{mainthm:A} applies. The reformulation of~(i)
  follows from the computation in the proof of
  Proposition~\ref{prop:mapping-space-geodesic}, which holds pointwise in
  $t \in \R$: $T(\evl_m)\bigl((\nabla^{K_\X}_\partial \dot\gamma)(t)\bigr)
  = (\nabla^{K_N}_\partial \dot\gamma_m)(t)$, and the family
  $(T(\evl_m))_{m \in M}$ is jointly injective.
\end{proof}

\begin{ex}\label{ex:mapping-flat}
  Let $N = \R^n$ with the Euclidean metric and the canonical full connection
  $(K_0, H)$ of Examples~\ref{ex:euclidean-vertical-connection}
  and~\ref{ex:euclidean-horizontal-connection}, whose geodesic flow is
  $\Phi(s, (u, v)) = u + sv$. Then Theorem~\ref{thm:mapping-variational}
  applies to $\X = C^\infty(M, \R^n)$, whose geodesics are the paths
  $t \mapsto f + t h$ with $f, h \in \X$. Write
  $\|h\|^2 := \int_M |h|^2 \vol_M$, so that the speed of a path $\sigma$ in
  $\X$ is $\|\partial_t \sigma(t)\|$. For $f_0, f_1 \in \X$, the geodesic
  $\gamma(t) = f_0 + t(f_1 - f_0)$ minimizes length: for every
  $\sigma \in \Path(\X; f_0, f_1)$, Minkowski's integral inequality gives
  \begin{align*}
    L(\sigma) = \int_0^1 \|\partial_t \sigma(t)\|\, dt
      \ge \Bigl\| \int_0^1 \partial_t \sigma(t)\, dt \Bigr\|
      = \|f_1 - f_0\| = L(\gamma).
  \end{align*}
  Thus, in contrast with Example~\ref{ex:torus-critical}, part~(ii) of
  Theorem~\ref{mainthm:A} is not vacuous here.
\end{ex}

Collecting the preceding propositions, we obtain the following theorem.

\begin{thm}\label{thm:summary-of-mapping-space}
Let $M$ be a closed manifold equipped with a volume form $\vol_M$ and let $(N,g_N)$ be an elastic Riemannian diffeological space such that $(M, T^kN)$ is tangent-commuting for every $k \ge 0$. Write $\X = C^\infty(M,N)$.
\begin{enumerate}
  \item $\X$ is elastic and its tangent structure is obtained pointwise from
    that of $N$.
  \item $g_{\mathit{map}}$ is a Riemannian metric on $\X$.
  \item If $K_N$ is a vertical connection on $N$, then $K_\X$ is a vertical connection on $\X$; $K_\X$ is torsion-free (resp.\ flat, effective) whenever $K_N$ is, and its curvature is $C_{K_N}\circ(-)$.
  \item If $K_N$ is a Levi-Civita connection for $g_N$, then $K_\X$ is a Levi-Civita connection for $g_{\mathit{map}}$.
  \item If $(K_N,H_N)$ is a full connection whose geodesic spray is complete, then the geodesic spray of $(K_\X,H_\X)$ is complete. If moreover $N \in \mathsf{Elast}_{\mathrm{curve}}$, every $v_0 \in T\X$ determines a unique geodesic.
  \item If moreover $K_N$ is a Levi-Civita connection for $g_N$ and $N \in \mathsf{Elast}_{\mathrm{curve}}$, then Theorem~\ref{mainthm:A} applies to $\X$ (Theorem~\ref{thm:mapping-variational}).
\end{enumerate}
\end{thm}


\begin{thebibliography}{99}
\bibitem[AB]{AB} L. Aintablian, C. Blohmann, Differentiable groupoid objects and their
  abstract Lie algebroids, Appl. Categ. Structures 33, no. 5, Paper No. 33, 97 pp., 2025
\bibitem[Bloh23]{B23} C. Blohmann, Elastic diffeological spaces, Recent Advances in Diffeologies and Their Applications, Contemp. Math. 794, AMS, 2024, pp. 49–86.
\bibitem[Bloh24]{B24} C. Blohmann, Lagrangian Field Theory, Diffeology, Variational Cohomology,
  Multisymplectic Geometry, 2024. {\tt https://people.mpim-bonn.mpg.de/blohmann/LFT/}
\bibitem[CC14]{CC14} J.R.B. Cockett and G.S.H. Cruttwell. Differential structure, tangent
  structure, and SDG. Appl. Categ. Structures, {\bf 22} (2014), 331--417.
\bibitem[CC15]{CC15} J.R.B. Cockett and G.S.H. Cruttwell. The Jacobi identity for tangent
  categories. Cah. Topol. G\'eom. Diff\'er. Cat\'eg. {\bf 56} (2015), 301--316.
\bibitem[CC17]{CC17} J.R.B. Cockett and G.S.H. Cruttwell, Connections in tangent categories,
  Theory and Applications of Categories, {\bf 32} (2017), 835--888.
\bibitem[CC18]{CC18} J.R.B. Cockett and G.S.H. Cruttwell, Differential bundles and fibrations
  for tangent categories, Cah. Topol. G\'eom. Diff\'er. Cat\'eg. 59.1 (2018), pp.\ 10--92.
\bibitem[CCL]{CCL} J.R.B. Cockett, G.S.H. Cruttwell, J.S.P. Lemay, Differential equations in
  a tangent category I: Complete vector fields, flows, and exponentials,
  Applied Categorical Structures 29 (5), 773--825, 2021.
\bibitem[GW]{GW} N. Goldammer, K. Welker, Towards optimization techniques on diffeological spaces by generalizing Riemannian concepts, Appl. Math. Optim. 93, No. 1, Paper No. 5, 31 p. (2026). 
\bibitem[Per17]{Per17} E. Pervova, Diffeological Levi-Civita connections, preprint. {\tt arXiv:1701.04988}.
\bibitem[PIZ13]{PIZ13} P. Iglesias-Zemmour, Diffeology, Mathematical Surveys and Monographs,
  185, AMS, Providence, 2013.
\bibitem[PIZ25]{PIZ25} P. Iglesias-Zemmour, Lectures on diffeology,
  Beijing World Publishing Corporation, 2025.
\bibitem[Kihara]{Kihara} H. Kihara, Smooth homotopy of infinite-dimensional $C^\infty$-manifolds,
  Memoirs of the American Mathematical Society 1436. Providence, RI: American Mathematical Society (AMS), 2023.
\bibitem[KMS]{KMS} I. Kol\'a\v{r}, P. W. Michor, J. Slov\'ak, Natural Operations in Differential Geometry, Springer-Verlag Berlin, 1993.
\bibitem[KSS]{KSS} K. Kuribayashi, K. Sakai, Y. Shiobara, Towards Riemannian diffeology,
  Proc. Roy. Soc. Edinburgh Sect. A, published online 2025, pp.\ 1--30.
  \texttt{doi:10.1017/prm.2025.10114}.
\bibitem[LW]{LW} R.B.B. Lucyshyn-Wright, On the geometric notion of connection and its
  expression in tangent categories. Theory and Applications of Categories 33 (2018), 832--866.
\bibitem[Miya]{M} D. Miyamoto, Lie algebras of quotient groups, preprint, 2025.
  {\tt https://arxiv.org/abs/2502.10260v2}.
\bibitem[Ros]{R} J. Rosick\'y. Abstract tangent functors. Diagrammes, 12: JR1--JR11, 1984.
\bibitem[Shio-Q]{ShioQ} Y. Shiobara, Riemannian structures on quotients in elastic diffeology via Q-charts, preprint, 2026. {\tt arXiv:2610.01346}.
\bibitem[Taho]{Taho} M. Taho, Tangent spaces of diffeological spaces and their variants,
Topology and its Applications, Volume 381, 2026, {\tt https://doi.org/10.1016/j.topol.2026.109741}.
\bibitem[Vin]{V} M. Vincent, Diffeological differential geometry, MSc. Copenhagen, DK: University of Copenhagen, 2008.  
\end{thebibliography}
\end{document}